\documentclass[reqno]{amsart}
\usepackage{amssymb}
\usepackage[usenames, dvipsnames]{color}

\usepackage{hyperref}

\usepackage{verbatim}

\theoremstyle{plain}
\newtheorem{theorem}{Theorem}[section]
\newtheorem{lemma}[theorem]{Lemma}
\newtheorem{corollary}[theorem]{Corollary}
\newtheorem{proposition}[theorem]{Proposition}

\theoremstyle{definition}

\newtheorem{assumption}[theorem]{Assumption}

\theoremstyle{remark}
\newtheorem{remark}[theorem]{Remark}

\numberwithin{equation}{section}

\newcommand{\bC}{\mathbb{C}}
\newcommand{\bN}{\mathbb{N}}

\newcommand{\bR}{\mathbb{R}}
\newcommand{\bH}{\mathbb{H}}
\newcommand{\bZ}{\mathbb{Z}}

\newcommand\cD{\mathcal{D}}

\newcommand\cG{\mathcal{G}}
\newcommand\cH{\mathcal{H}}

\newcommand\cM{\mathcal{M}}

\newcommand\cS{\mathcal{S}}

\makeatletter
\def\dashint{\operatorname%
{\,\,\text{\bf--}\kern-.98em\DOTSI\intop\ilimits@\!\!}}
\makeatother

\begin{document}
\title[parabolic equations with local and non-local time derivatives]{An approach for weighted mixed-norm estimates for parabolic equations with local and non-local time derivatives}

\author[H. Dong]{Hongjie Dong}
\address[H. Dong]{Division of Applied Mathematics, Brown University, 182 George Street, Providence, RI 02912, USA}

\email{Hongjie\_Dong@brown.edu}


\author[D. Kim]{Doyoon Kim}
\address[D. Kim]{Department of Mathematics, Korea University, 145 Anam-ro, Seongbuk-gu, Seoul, 02841, Republic of Korea}

\email{doyoon\_kim@korea.ac.kr}

\thanks{D. Kim was supported by the National Research Foundation of Korea (NRF) grant funded by the Korea government (MSIT) (2019R1A2C1084683).}



\begin{abstract}
We give a unified approach to weighted mixed-norm estimates and solvability for both the usual and time fractional parabolic equations in nondivergence form when coefficients are merely measurable in the time variable. In the spatial variables, the leading coefficients locally have small mean oscillations.
Our results extend the previous result in \cite{MR3899965} for unmixed $L_p$-estimates without weights.
\end{abstract}
\maketitle

\section{Introduction}

In this paper, we consider parabolic equations in nondivergence form
\begin{equation}
                                \label{eq7.53}
- \partial_t u + a^{ij}(t,x) D_{ij} u + b^i(t,x) D_i u + c(t,x) u = f(t,x)
\end{equation}
as well as parabolic equations with a non-local type time derivative term of the form
\begin{equation}
							\label{eq7.52}
- \partial_t^\alpha u + a^{ij}(t,x) D_{ij} u + b^i(t,x) D_i u + c(t,x) u = f(t,x)
\end{equation}
in $(0,T) \times \bR^d$, where $\partial_t^\alpha u$ is the Caputo fractional derivative of order $\alpha \in (0,1)$. See Section \ref{sec2} for the definition of $\partial_t^\alpha u$.

This paper is a continuation of \cite{MR3899965}, in which we proved that for any given $f \in L_p\left((0,T) \times \bR^d \right)$, there exists a unique solution $u$ to the equation \eqref{eq7.52} in $(0,T) \times \bR^d$ with the zero initial condition, and $u$ satisfies
$$
\||\partial_t^\alpha u|+|u|+
|Du|+|D^2u|\|_{L_p\left((0,T) \times \bR^d \right)} \le N \|f\|_{L_p\left((0,T) \times \bR^d \right)}
$$
under the assumptions that the coefficients are bounded and measurable, and $a^{ij}=a^{ij}(t,x)$ satisfy the uniform ellipticity condition and have small (bounded) mean oscillations (small BMO) with respect to the space variables.
Such type of coefficients were first introduced by Krylov in \cite{MR2304157} for \eqref{eq7.53} and the corresponding divergence form equations in $L_p$ spaces. His proof is based on the mean oscillation argument together with the Fefferman--Stein sharp function theorem and the Hardy--Littlewood maximal function theorem.
The results in \cite{MR2304157} were generalized later in \cite{MR2352490} to the mixed-norm $L_p(L_q)$ spaces with $p\ge q$, and in \cite{MR3812104} to the weighted mixed-norm $L_p(L_q)$ spaces with arbitrary $p, q\in (1,\infty)$ and Muckenhoupt weights.
In these papers, the mean oscillation argument is used, and in particular, in \cite{MR3812104} a version of the Fefferman--Stein sharp function theorem is proved in weighted mixed-norm Lebesgue spaces by using the extrapolation theorem.
We note that for time-dependent equations, such mixed-norm estimates are desirable, for example, when one wants to have better integrability of traces of solutions for each time slice when studying linear or nonlinear equations.

To the best of our knowledge, Equation \eqref{eq7.52} in mixed-norm Lebesgue spaces was first considered in \cite{MR3581300} for any $\alpha\in (0,2)$, under the conditions that the leading coefficients $a^{ij}$ are piecewise continuous in time and uniformly continuous in the spatial variables. Very recently, the result in \cite{arXiv:1911.07437} was extended to weighted mixed-norm Lebesgue spaces for the fractional heat equation
$$
-\partial_t^\alpha u+ \Delta u=f.
$$
The proofs in \cite{MR3581300} and  \cite{arXiv:1911.07437} are based upon a representation formula in terms of the fundamental solution to the time fractional heat operator $-\partial_t^\alpha + \Delta$ together with a perturbation argument using the Fefferman--Stein sharp function theorem and the Hardy--Littlewood maximal function theorem. On the other hand, the proof in \cite{MR3899965} (also see \cite{MR4030286}) is quite different from those in \cite{MR3581300, arXiv:1911.07437}. It is based upon a modified level set argument used in \cite{MR1486629} in order to improve the integrability of solution iteratively. In particular, we did not use the representation formula, which enables us to treat more general operators as those in \cite{MR3899965, MR4030286} with coefficients merely measurable in time or in one of the spatial variables.
For other results in this direction, we refer the reader to \cite{CP92, P91, Za05} and the references therein.

In view of the results in \cite{MR3812104} and \cite{MR3581300, arXiv:1911.07437}, it is natural to ask whether the result in \cite{MR3899965} can be extended to the mixed-norm $L_p(L_q)$ spaces and whether it is possible to also include weights. Unfortunately, it turns out that these extensions cannot be made by using the technique of iteration and the level set argument in \cite{MR3899965}. In this paper, we give a definite answer to these two questions. In particular, our main theorem reads that under the same assumptions on the coefficients as in \cite{MR3899965}, for any $p,q\in (1,\infty)$ and Muckenhoupt weights $w_1(t) \in A_p(\bR), w_2(x) \in A_q(\bR^d)$, if $u$ satisfies \eqref{eq7.52} with the zero initial condition, then it holds that
$$
\||\partial_t^\alpha u|+|u|+|Du|+|D^2 u|\|_{L_{p,q,w}\left((0,T)\times\bR^d\right)} \leq N \|f\|_{L_{p,q,w}\left((0,T)\times\bR^d\right)},
$$
where $w=w_1(t)w_2(x)$ and $N$ is independent of $u$ and $f$. See Section \ref{sec2} for the definition of the $L_{p,q,w}$-norm. Furthermore, we show that for any $f\in L_{p,q,w}\left((0,T)\times\bR^d\right)$, there is a unique solution $u$ (in the appropriate weighted mixed-norm Sobolev space) to \eqref{eq7.52} with the zero initial condition.

For the proof, we adapt the mean oscillation argument in \cite{MR2304157, MR2352490, MR3812104} mentioned above.
For this, we need to establish a H\"older estimate of $D^2 v$, where $v$ satisfies a certain homogeneous equation with coefficients depending only on $t$.
Such an estimate can be obtained relatively easily for parabolic equations with the local time derivative term $u_t$ via somewhat standard local estimates.
However, if the non-local time derivative is present, the local estimates do not work when improving the regularity of $v$ because the non-local time derivative of $v$ depends on all past states of $v$.
To overcome the difficulty from the non-local effect in time, our strategy is to consider an infinite cylinder $(-\infty,t_0)\times B_r(x_0)$ instead of the usual parabolic cylinder $Q_r(t_0,x_0)$ used in the aforementioned papers, and apply the Hardy--Littlewood maximal function for strong maximal functions. We estimate the H\"older semi-norm of $D^2 v$ by applying a bootstrap argument which relies on the (unmixed) $L_p$ estimate and the Sobolev type embedding results obtained in \cite{MR3899965}. For $w:=u-v$, which satisfies a nonhomogeneous equation also in the infinite cylinder with the zero Dirichlet boundary condition, we first apply the Poincar\'e inequality to bound $w$ by $f$, and then use a cutoff argument with a sequence of cutoff functions and again use the $L_p$ estimate in \cite{MR3899965} in order to estimate $D^2 w$. This gives a new decomposition of the solution, which works for both \eqref{eq7.53} and \eqref{eq7.52}. Since the argument for the usual parabolic equation \eqref{eq7.53} is less involved, even though the solvability result was already obtained in \cite{MR3812104}, to illustrate the idea in a simple setting we present the proof of the mean oscillation estimate for \eqref{eq7.53} in Section \ref{sec3}.
That is, in this paper we not only extend the results in \cite{MR3899965} to the weighted mixed-norm case, but also present a new and alternative approach to obtaining the mean oscillation estimates for both \eqref{eq7.53} and \eqref{eq7.52}.

The remaining part of the paper is organized as follows.
In the next section, we introduce some notation and state the main result of the paper. In Section \ref{sec3}, we show the mean oscillation estimate for \eqref{eq7.53} by using the new decomposition. In Section \ref{sec4}, we derive the corresponding  mean oscillation estimate for \eqref{eq7.52}.
Finally, we complete the proof of the main theorem in Section \ref{sec5}.

\section{Notation and main results}
                            \label{sec2}

We first introduce some notation used throughout the paper.
For $\alpha \in (0,1)$ and $S\in \bR$, we denote
$$
I^\alpha_S u = \frac{1}{\Gamma(\alpha)} = \frac{1}{\Gamma(\alpha)} \int_S^t (t-s)^{\alpha-1} u(s,x) \, ds.
$$
We write
$\partial_t^\alpha u=\partial_t u$ if $\alpha = 1$ and
$$
\partial_t^\alpha u = \frac{1}{\Gamma(1-\alpha)} \int_0^t (t-s)^{-\alpha} \partial_t u(s,x) \, ds
$$
for $\alpha \in (0,1)$. Note that $\partial_t^\alpha u = \partial_t I_0^\alpha u$ for a sufficiently smooth $u$ with $u(0,x) = 0$.
In some places we use $\partial_t^\alpha u$ to indicate $\partial_t I_S^\alpha u$ for $u(S,x) = 0$, where $S \neq 0$.

For $\alpha \in (0,1]$, we denote the parabolic cylinder by
\begin{equation}
							\label{eq0210_01}
Q_{r_1, r_2}(t,x) = (t-r_1^{2/\alpha}, t) \times B_{r_2}(x), \quad Q_r(t,x) = Q_{r,r}(t,x).
\end{equation}
We often write $B_r$ and $Q_r$ for $B_r(0)$ and $Q_r(0,0)$.
For $-\infty<S<T<\infty$, $\Omega \subset \bR^d$, we denote the parabolic boundary of the cylinder $(S,T) \times \Omega$ by
$$
\partial_p \left((S,T) \times \Omega\right) = \left((S,T) \times \partial \Omega\right) \cup \{(t,x) \in \bR^{d+1}: t=S, \, x \in \bar\Omega \}.
$$

For $p \in (1,\infty)$ and $k \in \{1,2,\ldots\}$, let $A_p(\bR^k, dx)$ be the set of all non-negative functions $w$ on $\bR^k$ such that
$$
[w]_{A_p}:= \sup_{x_0 \in \bR^k, r>0} \left(\dashint_{B_r(x_0)} w(x) \, dx \right) \left(\dashint_{B_r(x_0)} \left(w(x)\right)^{\frac{-1}{p-1}} \, dx \right)^{p-1} < \infty,
$$
where $B_r(x_0) = \{ x \in \bR^k: |x-x_0| < r\}$.
Recall that $[w]_{A_p} \geq 1$.

For $w(t,x) = w_1(t) w_2(x)$, where $(t,x) \in \bR \times \bR^d$, and $w_1 \in A_p(\bR,dt)$, $w_2 \in A_q(\bR^d, dx)$, we set $L_{p,q,w}\left((0,T) \times \bR^d\right)$ to be the set of all measurable functions $f$ defined on $(0,T) \times \bR^d$ satisfying
$$
\left(\int_0^T \left( \int_{\bR^d} |f(t,x)|^p \, w_2(x) dx \right)^{q/p} w_1(t) \, dt\right)^{1/q} < \infty.
$$
If $p=q$ and $w = 1$, $L_{p,q,w}\left((0,T) \times \bR^d\right)$ becomes the usual $L_p\left((0,T) \times \bR^d\right)$.
We write $u \in \bH_{p,q,w,0}^{\alpha,2}\left((0,T) \times \bR^d\right)$ if there exists a sequence $\{u_n\}$ such that $u_n \in C_0^\infty\left([0,T] \times \bR^d\right)$, $u_n(0,x) = 0$, and
\begin{equation}
							\label{eq0316_01}
\|u_n - u\|_{p,q,w} + \|Du_n - Du\|_{p,q,w} + \|D^2u_n - D^2u\|_{p,q,w} + \|\partial_t^\alpha u_n - \partial_t^\alpha u\|_{p,q,w} \to 0
\end{equation}
as $n \to \infty$, where $\|\cdot\|_{p,q,w} = \|\cdot\|_{L_{p,q,w}\left((0,T) \times \bR^d\right)}$.
We often write $\bH_{p,w,0}^{\alpha,2}\left((0,T) \times \bR^d\right)$ if $p = q$.
In particular, we see that if $\alpha = 1$, $p=q$, and $w(t,x) = 1$, then
$$
\bH_{p,0}^{1,2}\left((0,T) \times \bR^d\right) = \{u \in W_p^{1,2}\left((0,T) \times \bR^d\right): u(0,x) = 0\},
$$
where
$$
W_p^{1,2}\left((0,T) \times \bR^d\right)  = \{u: u,Du,D^2u,\partial_t u \in L_p\left((0,T) \times \bR^d\right)\}.
$$
However, for $\alpha \in (0,1)$, as remarked in \cite[Remark 3.4]{MR3899965},
\begin{align*}
\bH_{p, 0}^{\alpha,2}\left((0,T) \times \bR^d\right) &\subsetneq \bH_p^{\alpha,2}\left((0,T) \times \bR^d\right)
\\
&\subsetneq \{u: u, Du, D^2u, \partial_t^\alpha u \in L_p\left((0,T) \times \bR^d\right)\}.
\end{align*}
See \cite{MR3899965} for more details about $\bH_p^{\alpha,2}((0,T) \times \Omega)$ and $\bH_{p,0}^{\alpha,2}((0,T) \times \Omega)$, $\Omega \subset \bR^d$.
We use the notation $u \in \bH_{p,0, \operatorname{loc}}^{\alpha,2}\left((0,T) \times \bR^d \right)$ to indicate that
$$
u \in \bH_{p,0}^{\alpha,2}\left((0,T) \times B_R \right)
$$
for any $R > 0$.
In particular, when $\alpha = 1$, by $u \in W_{p,\operatorname{loc}}^{1,2}\left((-\infty,T) \times \bR^d\right)$ we mean that $u \in W_p^{1,2}\left((-\infty,T) \times B_R\right)$ for any $R > 0$.
Likewise, $f \in L_{p,\operatorname{loc}}\left((0,T) \times \bR^d\right)$ means that $f \in L_p\left((0,T) \times B_R\right)$ for any $R > 0$.

In this paper, we use maximal functions and strong maximal functions defined, respectively, by
$$
\cM f (t,x) =  \sup_{Q_r(s,y) \ni (t,x)} \dashint_{Q_r(s,y)} |f(r,z)| \chi_\cD \, dz \, dr
$$
and
$$
\left(\cS \cM f\right) (t,x) = \sup_{Q_{r_1,r_2}(s,y) \ni (t,x)} \dashint_{Q_{r_1,r_2}(s,y)} |f(r,z)| \chi_\cD \, dz \, dr
$$
if $f$ is defined on $\cD \subset \bR^{d+1}$.

We now present our assumptions for the coefficients. Throughout the paper we assume that there exists $\delta \in (0,1)$ such that
\begin{equation}
							\label{eq0131_20}
a^{ij}(t,x)\xi_i \xi_j \geq \delta |\xi|^2, \quad | a^{ij}| \leq \delta^{-1}
\end{equation}
for any $(t,x) \in \bR^{d+1}$ and $\xi \in \bR^d$.
Without loss of generality, we assume that $a^{ij} = a^{ji}$.

We impose the following regularity assumption on $a^{ij}(t,x)$ with respect to $x \in \bR^d$.

\begin{assumption}[$\gamma_0$]
							\label{assum0210_1}
There is a constant $R_0 \in (0,1]$ such that, for each parabolic cylinder $Q_r(t_0,x_0) = (t_0 - r^{2/\alpha}, t_0) \times B_r(x_0)$ with $r \leq R_0$ and $(t_0,x_0) \in \bR^{d+1}$, we have
$$
\sup_{i,j} \dashint_{Q_r(t_0,x_0)} |a^{ij} - \bar{a}^{ij}(t)| \, dx \, dt \leq \gamma_0,
$$
where
$$
\bar{a}^{ij}(t) = \dashint_{B_r(x_0)} a^{ij}(t,y) \, dy.
$$
\end{assumption}

For the lower-order coefficients $b^i$ and $c$, we impose the following boundedness condition
$$
|b^i| \leq K_0, \quad |c| \leq K_0,
$$
where $K_0$ is a positive constant.

Here is the main theorem of the paper.

\begin{theorem}
							\label{thm0210_1}
Let $\alpha \in (0,1]$, $T \in (0,\infty)$, $p, q \in (1,\infty)$, $K_1 \in [1,\infty)$, $w = w_1(t) w_2(x)$, where
$$
w_1(t) \in A_p(\bR,dt), \quad w_2(x) \in A_q(\bR^d, dx), \quad [w_1]_{A_p} \leq K_1, \quad [w_2]_{A_q} \leq K_1.
$$
There exists $\gamma_0 = \gamma_0(d,\delta, \alpha, p,q,K_1) \in (0,1)$ such that, under Assumption \ref{assum0210_1} ($\gamma_0$), the following hold.

For any $u \in \bH_{p, q, w, 0}^{\alpha,2}\left((0,T) \times \bR^d\right)$ satisfying
\begin{equation}
							\label{eq0210_02}
- \partial_t^{\alpha} u + a^{ij} D_{ij} u + b^i D_i u + c u  = f
\end{equation}
in $(0,T) \times \bR^d$, we have
\begin{equation}
							\label{eq0226_01}
\|\partial_t^\alpha u\|_{p,q,w} + \|u\|_{p,q,w} + \|Du\|_{p,q,w} + \|D^2 u\|_{p,q,w} \leq N \|f\|_{p,q,w},
\end{equation}
where $\|\cdot\|_{p,q,w} = \|\cdot\|_{L_{p,q,w}\left((0,T)\times\bR^d\right)}$ and
$$
N = N(d,\delta, \alpha, p,q,K_1,K_0,R_0,T).
$$
Moreover, for any $f \in L_{p,q,w}\left((0,T) \times \bR^d\right)$, there exists a unique solution $u \in \bH_{p,q,w,0}^{\alpha,2}\left((0,T) \times \bR^d)\right)$ satisfying \eqref{eq0210_02}.
\end{theorem}

\section{Mean oscillation estimates for equations with local time derivative}
                    \label{sec3}

Throughout the section we assume that $a^{ij} = a^{ij}(t)$, that is, functions of only $t$, satisfying the ellipticity condition \eqref{eq0131_20}.
Let $T \in (-\infty,\infty]$ and $1 < p_0 < \infty$. Let $u \in W_{p_0,\operatorname{loc}}^{1,2}\left((-\infty,T) \times \bR^d\right)$ be a solution to
\begin{equation}
							\label{eq0131_01}
- \partial_t u + a^{ij} D_{ij} u = f
\end{equation}
in $(-\infty,T) \times \bR^d$.
For $t_0\in (-\infty, T] \cap \bR$ and $r>0$, we decompose $u=v+w$ in $(-\infty,t_0)\times B_r$, where $w \in W_{p_0}^{1,2}\left((-\infty,t_0) \times B_r\right)$ satisfies
\begin{equation}
							\label{eq0131_02}
- \partial_tw + a^{ij}(t) D_{ij} w = f
\end{equation}
in $(-\infty,t_0)\times B_r$ with the zero boundary condition on $(-\infty,t_0)\times \partial B_r$, and $v \in W_{p_0}^{1,2}\left((-\infty,t_0) \times B_r\right)$ satisfies
\begin{equation}
							\label{eq0131_03}
- \partial_t v + a^{ij}(t) D_{ij} v = 0
\end{equation}
in $(-\infty,t_0)\times B_r$.

In this section by $Q_r(t_0,x_0)$ we mean the parabolic cylinder defined in \eqref{eq0210_01} with $\alpha = 1$. That is,
$$
Q_r(t_0,x_0) = (t_0-r^2,t_0) \times B_r(x_0).
$$
By using the usual iteration argument, we obtain the following estimate for $v$ satisfying \eqref{eq0131_03}.
See, for instance, \cite[Theorems 5.2.8 and 6.1.1]{MR2435520} or \cite[Lemma 5.6]{MR3812104}.

\begin{lemma}
							\label{lem0131_3}
Let $p_0\in (1,\infty)$, $t_0 \in \bR$, and $v\in W_{p_0}^{1,2}((-\infty,t_0)\times B_r)$ satisfy \eqref{eq0131_03} in $(-\infty,t_0) \times B_r$, $r > 0$.
Then
\begin{equation*}
[D^2v]_{C^{1/4,1/2}(Q_{r/2}(t_1,0))} \leq N r^{-1/2}
\left( |D^2v|^{p_0} \right)_{Q_r(t_1,0)}^{1/p_0}
\end{equation*}
for any $t_1 \leq
t_0$, where $N = N(d, \delta, p_0)$.
\end{lemma}

For $w$ satisfying \eqref{eq0131_02}, we bound the $L_{p_0}$-norm of $D^2w$ on $Q_{1/2}$ by the sum of the averages of $f$ on cylinders of the form $Q_{r,1}$, which is in fact bounded by the strong maximal function of $|f|^{p_0}$.
To obtain such an estimate, we first bound the $L_{p_0}$-norm of $w$ on $Q_1$ by the sum of such averages.

\begin{lemma}
                            \label{lem0131_2}
Let $p_0\in (1,\infty)$, $t_0 \in \bR$, and $f \in L_{p_0}\left((-\infty,t_0) \times B_1\right)$. Let $w\in W_{p_0}^{1,2}((-\infty,t_0)\times B_1)$ be a solution to \eqref{eq0131_02} in $(-\infty, t_0) \times B_1$ with the zero boundary condition on $(-\infty,t_0) \times \partial B_1$.
Then we have
\begin{equation}
							\label{eq0212_01}
\|w\|_{L_{p_0}(Q_1(t_0,0))} \leq N \sum_{k=0}^\infty c_k \left(|f|^{p_0}\right)_{(t_0-2^{k+2}+2, t_0) \times B_1}^{1/p_0},
\end{equation}
where $N=N(d,\delta,p_0)$ is independent of $t_0$ and $\{c_k\}$ is a sequence satisfying
\begin{equation}
							\label{eq0212_04}
\sum_{k=0}^\infty c_k\le N=N(d,\delta,p_0).
\end{equation}
\end{lemma}

\begin{proof}
To derive \eqref{eq0212_01}, we prove that there exists $\varepsilon = \varepsilon(d,\delta,p_0)>0$ such that
\begin{equation}
							\label{eq0203_02}
e^{\varepsilon(t_0-1)} \|w\|_{L_{p_0}(Q_1(t_0,0))} \leq N \|e^{\varepsilon t} f(t,x)\|_{L_{p_0}((-\infty,t_0)\times B_1)},
\end{equation}
where $N = N(d,\delta,p_0)$.
Note that
\begin{align*}
&\|e^{\varepsilon t} f(t,x)\|_{L_{p_0}((-\infty,t_0)\times B_1)} = \left(\sum_{k=0}^\infty \int_{t_0-2^{k+2}+2}^{t_0-2^{k+1}+2} e^{\varepsilon p_0 t} \int_{B_1} |f(t,x)|^{p_0} \, dx \, dt\right)^{1/p_0}
\\
&\qquad\leq \left(\sum_{k=0}^\infty \int_{t_0-2^{k+2}+2}^{t_0-2^{k+1}+2} e^{\varepsilon p_0 (t_0 - 2^{k+1} + 2)}\int_{B_1} |f(t,x)|^{p_0} \, dx \, dt \right)^{1/p_0}
\\
&\qquad\leq N(d) e^{\varepsilon t_0} \left(\sum_{k=0}^\infty e^{\varepsilon p_0 (2 - 2^{k+1})} (2^{k+2}-2) \left(|f|^{p_0}\right)_{(t_0-2^{k+2}+2,t_0) \times B_1}\right)^{1/p_0}
\\
&\qquad\leq N e^{\varepsilon t_0} \sum_{k=0}^\infty e^{\varepsilon(2-2^{k+1})} (2^{k+2}-2)^{1/p_0}\left(|f|^{p_0}\right)_{(t_0-2^{k+2}+2,t_0) \times B_1}^{1/p_0}.
\end{align*}
Thus, by setting
$$
c_k = e^{\varepsilon (2-2^{k+1})} (2^{k+2} - 2)^{1/p_0},
$$
which satisfies \eqref{eq0212_04}, from \eqref{eq0203_02} we arrive at \eqref{eq0212_01}.

For the proof of \eqref{eq0203_02}, we show that there exists $\varepsilon_0 = \varepsilon_0(d,\delta,p_0)>0$ such that, for $\varepsilon \in [0,\varepsilon_0]$, $U \in W_{p_0}^{1,2}\left((-\infty,\tau_0) \times B_1\right)$, and $F \in L_{p_0}\left((-\infty,\tau_0) \times B_1\right)$ satisfying
\begin{equation}
							\label{eq0203_03}
- \partial_t U + a^{ij}(t) D_{ij} U + \varepsilon U = F
\end{equation}
in $(-\infty,\tau_0) \times B_1$ and $U=0$ on $(-\infty,\tau_0) \times \partial B_1$, where $\tau_0 \in (-\infty,\infty]$,
we have
\begin{equation}
							\label{eq0203_04}
\|U\|_{L_{p_0}\left((-\infty,\tau_0) \times B_1\right)} \leq N \|F\|_{L_{p_0}\left((-\infty,\tau_0) \times B_1\right)},
\end{equation}
where $N = N(d,\delta,p_0)$ is independent of $\tau_0$.
If this holds, one can obtain \eqref{eq0203_02} by taking $\varepsilon \in [0, \varepsilon_0]$, $\tau_0 = t_0$,
$$
U = w e^{\varepsilon t} \in W_{p_0}^{1,2}\left((-\infty,t_0) \times B_1\right), \quad F = e^{\varepsilon t} f \in L_{p_0}\left((-\infty,t_0) \times B_1\right).
$$

To show \eqref{eq0203_04}, we consider two cases.

{\em Case 1: $p_0\in [2,\infty)$.}
We multiply both sides of \eqref{eq0203_03} by $-p_0 |U|^{p_0-2} U$ and integrate in $(-\infty,\tau_0)\times B_1$.
By noting that the integral involving $\partial_t U$ is nonnegative, we have
$$
\int_{-\infty}^{\tau_0} \int_{B_1}
a_{ij}p_0(p_0-1)|U|^{p_0-2} D_j U D_i U \,dx \, dt - \varepsilon p_0 \int_{-\infty}^{\tau_0} \int_{B_1} |U|^{p_0} \, dx \, dt
$$
$$
\le \int_{-\infty}^{\tau_0} \int_{B_1} p_0 |F||U|^{p_0-1} \,dx \, dt.
$$
By using the zero boundary condition and the Poincar\'e inequality to $G(U) = |U|^{p_0/2}$ with respect to the spatial variables, as well as using the ellipticity condition, from the above inequality we get
\begin{align*}
&\int_{-\infty}^{\tau_0} \int_{B_1} |U|^{p_0} \, dx \, dt = \int_{-\infty}^{\tau_0} \int_{B_1} |G(U)|^2 \, dx \, dt
\\
&\leq N \int_{-\infty}^{\tau_0} \int_{B_1} |D \left(G(U)\right)|^2 \, dx \,dt
\leq N \int_{-\infty}^{\tau_0} \int_{B_1} \frac{p_0^2}{4} |DU|^2 |U|^{p_0-2} \, dx \, dt
\\
&\leq N \int_{-\infty}^{\tau_0} \int_{B_1} p_0(p_0-1)|DU|^2 |U|^{p_0-2} \, dx \, dt
\\
&\leq N \int_{-\infty}^{\tau_0} \int_{B_1} e^{\varepsilon t} |F||U|^{p_0-1} \, dx \,dt + N \varepsilon \int_{-\infty}^{\tau_0} \int_{B_1} |U|^{p_0} \, dx \,dt,
\end{align*}
where $N = N(d,\delta,p_0)$.
Upon choosing a sufficiently small $\varepsilon_0 = \varepsilon_0(d,\delta,p_0) > 0$, from the above inequalities and Young's inequality, we arrive at \eqref{eq0203_04} whenever $\varepsilon \in [0,\varepsilon_0]$.

{\em Case 2: $p_0\in (1,2)$.} We use a duality argument.
Take $\varepsilon_0$ from Case 1 and set $q_0=p_0/(p_0-1)\in (2,\infty)$.
For $\varepsilon \in [0,\varepsilon_0]$ and $g\in C^\infty_0 \left((-\infty,\tau_0)\times B_1 \right)$, let $U_1$ be the unique solution in $W_{q_0}^{1,2}\left((-\tau_0,\infty) \times B_1\right)$ of the equation
$$
-\partial_t U_1+a^{ij}(-t)D_{ij}U_1 + \varepsilon U_1 =g(-t,x)
$$
in $(-\tau_0,\infty)\times B_1$ with the zero boundary condition on $\partial_p \left( (-\tau_0,\infty) \times B_1 \right)$.
Indeed, to obtain the existence of such a solution $U_1$, we solve
\begin{equation}
							\label{eq0203_05}
- \partial_t U_1 + a^{ij}(-t) D_{ij} U_1 + \varepsilon U_1 = g(-t,x) \chi_{(-t_0,\infty)}
\end{equation}
in $\bR \times B_1$ with the zero boundary condition on $\bR \times \partial B_1$.
The solvability of this equation is guaranteed by the usual $L_p$-theory for parabolic equations with coefficients measurable in time (see, for instance, \cite{MR2771670, MR2644213}) and the estimate
\begin{equation}
							\label{eq0203_06}
\|U_1\|_{L_{q_0}\left((-\tau_0, \infty)\times B_1\right)} \leq N \|g(-\cdot,\cdot)\|_{L_{q_0}\left((-\tau_0, \infty)\times B_1\right)},
\end{equation}
proved in Case 1 when the time interval is $\bR$.
One can check the zero initial condition of $U_1$ at $t = -\tau_0$ by noting that the solution $U_1$ to the equation \eqref{eq0203_05} is zero for $t \leq -\tau_0$.
Then by using integration by parts, we have
$$
\int_{-\infty}^{t_0} \int_{B_1} U g  \,dx \, dt = \int_{-\infty}^{t_0} \int_{B_1} F(t,x) U_1(-t,x) \, dx \, dt.
$$
This together with \eqref{eq0203_06} immediately gives \eqref{eq0203_04}.
The lemma is proved.
\end{proof}

\begin{lemma}
                    \label{lem0203_1}
Let $p_0\in (1,\infty)$, $t_0 \in \bR$, and $f \in L_{p_0}\left( (-\infty,t_0) \times B_1 \right)$. Let $w\in W_{p_0}^{1,2}((-\infty,t_0)\times B_1)$ be a solution to \eqref{eq0131_02} in $(-\infty, t_0) \times B_1$ with the zero boundary condition on $(-\infty,t_0) \times \partial B_1$.
Then we have
\begin{equation}
							\label{eq0212_02}
\left( |D^2 w|^{p_0} \right)_{Q_{1/2}(t_1,0)}^{1/p_0} \le N \sum_{k=0}^\infty c_k \left( |f|^{p_0} \right)_{(t_0-2^{k+2}+2,t_0) \times B_1}^{1/p_0},
\end{equation}
where $N = N(d,\delta,p_0)$ and $\{c_k\}$ satisfies \eqref{eq0212_04}.
\end{lemma}

\begin{proof}
Since the assumptions are satisfied if $t_0$ is replaced with $t_1$, we prove only the case $t_1 = t_0$.
By using the standard local $L_{p_0}$-estimate for parabolic equations with coefficients measurable in time (see, for instance, \cite[Theorem 6.4.2]{MR2435520}), we have
\begin{equation*}
\left( |D^2 w|^{p_0} \right)_{Q_{1/2}(t_0,0)}^{1/p_0}
\le N\left( |f|^{p_0} \right)_{Q_1(t_0,0)}^{1/p_0}
+ N (|w|^{p_0})_{Q_1(t_0,0)}^{1/p_0},
\end{equation*}
where $N = N(d,\delta, p_0)$.
This combined with Lemma \ref{lem0131_2} proves \eqref{eq0212_02}. The lemma is proved.
\end{proof}

\begin{proposition}
							\label{prop0204_1}
Let $p_0 \in (1,\infty)$, $T \in (-\infty,\infty]$, and $u \in W_{p_0, \operatorname{loc}}^{1,2}\left((-\infty,T) \times \bR^d\right)$ satisfy \eqref{eq0131_01} in $(-\infty,T) \times \bR^d$.
Then, for any $(t_0,x_0) \in (-\infty, T] \times \bR^d$ with $t_0 \in \bR$, $r \in (0,\infty)$, and $\kappa \in (0, 1/4)$,
we have
\begin{equation*}
\begin{aligned}
&\left(|D^2u - (D^2u)_{Q_{\kappa r}(t_0,x_0)}| \right)_{Q_{\kappa r}(t_0,x_0)}
\\
&\leq N \kappa^{1/2} (|D^2u|^{p_0})_{Q_r(t_0,x_0)}^{1/p_0} + N \kappa^{-(d+2)/p_0} \sum_{k=0}^\infty c_k \left( |f|^{p_0} \right)_{(t_0-(2^{k+2}-2)r^2,t_0) \times B_r(x_0)}^{1/p_0},
\end{aligned}
\end{equation*}
where $N = N(d,\delta,p_0)$ and $\{c_k\}$ satisfies \eqref{eq0212_04}.
\end{proposition}

\begin{proof}
Thanks to translation with respect to the spatial variables and dilation, we assume that $x_0 = 0$ and $r=1$.
Since $u \in W_{p_0,\operatorname{loc}}^{1,2}\left((-\infty,T) \times \bR^d\right)$, we have $f \in L_{p_0}\left((-\infty,T) \times B_1\right)$. Thus one can find $w \in W_{p_0}^{1,2}\left((-\infty,t_0) \times B_1\right)$ satisfying \eqref{eq0131_02} in $(-\infty,t_0) \times B_1$ with the zero boundary condition on $(-\infty,t_0) \times \partial B_1$.
We note that the solvability follows from \eqref{eq0203_04}, the interior and boundary $W^{1,2}_{p}$-estimates (see, for instance, \cite{MR2771670}), and the method of continuity.
Set $v = u - w$.
Then $v$ belongs to $W_{p_0}^{1,2}\left((-\infty,t_0) \times B_1\right)$ and satisfies \eqref{eq0131_03} in $(-\infty,t_0) \times B_1$.
Since $\kappa < 1/4$, by Lemma \ref{lem0131_3} with $r = 1/2$ and the fact that $u=v+w$, we observe that
\begin{align*}
&\left(|D^2v - (D^2v)_{Q_{\kappa}(t_0,0)}| \right)_{Q_{\kappa}(t_0,0)} \leq 3 \kappa^{1/2} [D^2v]_{C^{1/4,1/2}\left(Q_{1/4}(t_0,0)\right)}\\
&\leq N \kappa^{\frac{1}{2}} \left(|D^2 v|^{p_0}\right)^{\frac{1}{p_0}}_{Q_{\frac 1 2}(t_0,0)} \leq N \kappa^{\frac{1}{2}} \left(|D^2 u|^{p_0}\right)^{\frac{1}{p_0}}_{Q_{\frac{1}{2}}(t_0,0)} + N \kappa^{\frac{1}{2}} \left(|D^2 w|^{p_0}\right)^{\frac{1}{p_0}}_{Q_{\frac{1}{2}}(t_0,0)}.
\end{align*}
This combined with Lemma \ref{lem0203_1} and the triangle inequality shows that
\begin{align*}
&\left(|D^2u - (D^2u)_{Q_{\kappa}(t_0,0)}| \right)_{Q_{\kappa}(t_0,0)}
\\
&\leq \left(|D^2v - (D^2v)_{Q_{\kappa}(t_0,0)}| \right)_{Q_{\kappa}(t_0,0)} + N \left(|D^2w| \right)_{Q_{\kappa}(t_0,0)}
\\
&\leq N \kappa^{\frac{1}{2}} \left(|D^2 u|^{p_0}\right)^{\frac{1}{p_0}}_{Q_{\frac{1}{2}}(t_0,0)}
+ N \kappa^{-\frac{d+2}{p_0}} \left(|D^2w|^{p_0} \right)^{\frac{1}{p_0}}_{Q_{\frac{1}{2}}(t_0,0)}
\\
&\leq N \kappa^{\frac{1}{2}} \left(|D^2 u|^{p_0}\right)^{\frac{1}{p_0}}_{Q_1(t_0,0)} + N \kappa^{-\frac{d+2}{p_0}} \sum_{k=0}^\infty c_k \left( |f|^{p_0} \right)_{(t_0-2^{k+2}+2,t_0) \times B_1}^{1/p_0},
\end{align*}
where $N = N(d,\delta,p_0)$.
The proposition is proved.
\end{proof}

\section{Mean oscillation estimates for equations with non-local time derivative}
                    \label{sec4}

Throughout the section we assume that $\alpha\in (0,1)$ and $a^{ij} = a^{ij}(t)$, that is, functions of only $t$, satisfying the ellipticity condition \eqref{eq0131_20}.
Let $ p_0\in (1,\infty)$ and $T \in (0,\infty)$.
Let $u \in \bH_{p_0,0,\operatorname{loc}}^{\alpha,2}\left((0,T) \times \bR^d\right)$ be a solution to
\begin{equation}
							\label{eq0525_01}
- \partial_t^\alpha u + a^{ij}(t) D_{ij} u= f(t,x)
\end{equation}
in $(0,T) \times \bR^d$.
Note that the zero initial condition $u(0,\cdot) = 0$ is implicitly imposed because $u \in \bH_{p_0,0,\operatorname{loc}}^{\alpha,2}\left((0,T) \times \bR^d\right)$.

For convenience, we extend $u$ and $f$ to be zero when $t\le 0$. For $t_0\in (0,T]$ and $r>0$, we decompose $u=v+w$ in $(0,t_0)\times B_r$, where $w$ is a weak solution to
\begin{equation}
							\label{eq1.01}
- \partial_t^\alpha w + a^{ij}(t) D_{ij} w= f(t,x)
\end{equation}
in $(0,t_0) \times B_r$ with the zero boundary condition on $\partial_p ((0,t_0)\times B_r)$, and $v$ satisfies
\begin{equation}
							\label{eq1.03}
- \partial_t^\alpha v + a^{ij}(t) D_{ij} v=0
\end{equation}
in $(0,t_0) \times B_r$.
Throughout the section, recall that, for $\alpha \in (0,1)$,
$$
Q_r(t_0,x_0) = (t_0 - r^{2/\alpha}, t_0) \times B_r(x_0).
$$

\subsection{Estimates of \texorpdfstring{$v$}{v}}

\begin{lemma}
                    \label{lem1}
Let $p_0\in (1,\infty)$, $t_0 \in (0,\infty)$, and $v\in \bH^{\alpha,2}_{p_0,0}((0,t_0)\times B_r)$ satisfy \eqref{eq1.03} in $(0,t_0) \times B_r$, $r > 0$. Then there exists
$$
p_1 = p_1(d, \alpha,p_0)\in (p_0,\infty]
$$
satisfying
\begin{equation}
							\label{eq0411_05}
p_1 > p_0 + \min\left\{\frac{2\alpha}{\alpha d + 2 - 2\alpha}, \alpha, \frac{2}{d} \right\}
\end{equation}
such that
\begin{equation}
							\label{eq0108_01}
\left( |D^2v|^{p_1} \right)_{Q_{r/2}(t_1,0)}^{1/p_1} \leq N
\sum_{j=1}^\infty j^{-(1+\alpha)} \left( |D^2v|^{p_0} \right)_{Q_r(t_1-(j-1)r^{2/\alpha},0)}^{1/p_0}
\end{equation}
for any $t_1 \leq t_0$,
where $N=N(d,\delta, \alpha,p_0)$ and
$$
\left( |D^2v|^{p_1} \right)_{Q_{r/2}(t_1,0)}^{1/p_1} = \|D^2 v\|_{L_\infty\left(Q_{r/2}(t_1,0)\right)} \quad \text{if} \quad p_1 = \infty.
$$
If $p_0 > d/2+1/\alpha$, then
\begin{equation}
							\label{eq0108_02}
[D^2 v]_{C^{\sigma \alpha/2,\sigma}(Q_{r/2}(t_1,0))} \leq N r^{-\sigma}
\sum_{j=1}^\infty j^{-(1+\alpha)} \left( |D^2v|^{p_0} \right)_{Q_r(t_1-(j-1)r^{2/\alpha},0)}^{1/p_0}
\end{equation}
for any $t_1 \leq
t_0$, where $\sigma = \sigma(d,\alpha,p_0) \in (0,1)$.
Moreover, if $p_1 < \infty$, then $v \in \bH_{p_1,0}^{\alpha,2}\left((0,t_0) \times  B_{r/2}\right)$.
\end{lemma}

\begin{proof}
Thanks to scaling, we only need to prove the assertions for $r=1$.
Note that $v$ can be extended as zero for $t \leq 0$.
Thus by Lemma 3.5 in \cite{MR3899965} $v$ belongs to $\bH_{p_0,0}^{\alpha,2} \left((S,t_0) \times B_1\right)$ for any $S \leq 0$.
We also see that $v$ satisfies \eqref{eq1.03} in $(S,t_0) \times B_1$.

Find an infinitely differentiable  function $\eta$ defined on $\bR$ such that
$$
\eta =
\left\{
\begin{aligned}
1 \quad &\text{if} \quad t \in (t_1-(1/2)^{2/\alpha},t_1),
\\
0 \quad &\text{if} \quad t \in \bR \setminus (t_1-1,t_1+1),
\end{aligned}
\right.
$$
and
$$
\left|\frac{\eta(t)-\eta(s)}{t-s}\right| \le N(\alpha).
$$
Then by Lemmas 3.6 and 4.4 (and the proof of the latter lemma)
in \cite{MR3899965},
\begin{equation}
							\label{eq0109_01}
\eta v, D(\eta v),
D^2(\eta v) \in \bH_{p_0,0}^{\alpha,2}\left((t_1-1,t_1)\times B_{3/4}\right)
\end{equation}
and $D^2(\eta v)$ satisfies
\begin{equation}
                \label{eq9.43}
-\partial_t^\alpha \left( D^2(\eta v) \right) + a^{ij} D_{ij} D^2(\eta v) = \cG
\end{equation}
in $(t_1-1,t_1)\times B_{3/4}$, where $\partial_t^\alpha = \partial_t I^{1-\alpha}_{t_1-1}$ and
$$
\cG(t,x) = \frac{\alpha}{\Gamma(1-\alpha)} \int_{-\infty}^t (t-s)^{-\alpha-1}\left(\eta(t) - \eta(s)\right) D^2 v(s,x) \, ds.
$$

If $p \leq 1/\alpha$, take $p_1$ satisfying
$$
p_1 \in \left(p_0, \frac{1/\alpha + d/2}{1/(\alpha p_0) + d/(2 p_0) -1}\right) \quad \text{if} \quad p_0 \leq d/2,
$$
$$
p_1 \in \left(p_0, p_0(\alpha p_0 + 1)\right) \quad \text{if} \quad p_0 > d/2.
$$
If $p_0 > 1/\alpha$, take $p_1$ satisfying
$$
p_1 \in \left(p_0, p_0 + 2p_0^2/d\right) \quad \text{if} \quad p_0 \leq d/2,
$$
$$
p_1 \in (p_0, 2p_0) \quad \text{if} \quad p_0 > d/2, \quad p_0 \leq d/2 + 1/\alpha,
$$
$$
p_1 = \infty \quad \text{if} \quad p_0 > d/2 + 1/\alpha.
$$
Note that $p_1$ satisfies \eqref{eq0411_05}
and the increment $\min \{2\alpha/(\alpha d + 2 - 2\alpha), \alpha, 2/d\}$ is independent of $p_0$.
By \eqref{eq0109_01} and the Sobolev type embeddings  obtained in \cite{MR3899965} we have
\begin{equation}
							\label{eq0109_02}
\eta v, D(\eta v),
D^2(\eta v) \in L_{p_1}\left((t_1-1,t_1) \times B_{3/4}\right).
\end{equation}
By the Sobolev embeddings again and Lemma 4.2 in \cite{MR3899965} for local $L_p$-estimates, we have
\begin{align}
							\label{eq0715_01}
&\|D^2 v\|_{L_{p_1}\left(Q_{1/2}(t_1,0)\right)} \le \|D^2(\eta v)\|_{L_{p_1}\left((t_1-1,t_1)\times B_{1/2}\right)}\nonumber
\\
&\le N\| |D^2(\eta v)| + |D^4(\eta v)| + |D_t^\alpha D^2 (\eta v)| \|_{L_{p_0}\left((t_1-1,t_1)\times B_{3/4}\right)}\nonumber
\\
&\le N \||D^2(\eta v)| + |\cG|\|_{L_{p_0}\left((t_1-1,t_1)\times B_1\right)} \le N \||D^2 v| + |\cG|\|_{L_{p_0}\left((t_1-1,t_1)\times B_1\right)},
\end{align}
where $N = N(d, \delta,
\alpha,p_0,p_1)$ and we used \eqref{eq9.43}.
We write
\begin{align*}
&\frac{\Gamma(1-\alpha)}{\alpha} \cG(t,x) = \int_{t-1}^t (t-s)^{-\alpha-1}\left( \eta(s) - \eta(t) \right) D^2 v(s,x) \, ds\\
&\quad + \int_{-\infty}^{t-1} (t-s)^{-\alpha-1}\left( \eta(s) - \eta(t) \right) D^2 v(s,x) \, ds := I_1(t,x)+I_2(t,x),
\end{align*}
where
\begin{align*}
|I_1(t,x)| \le N \int_{t-1}^t |t-s|^{-\alpha} |D^2 v(s,x)|\,ds= N \int_0^1 |s|^{-\alpha} |D^2 v(t-s,x)|\, ds,
\end{align*}
which implies
\begin{equation}
							\label{eq0715_02}
\|I_1\|_{L_{p_0}\left((t_1-1,t_1) \times B_1\right)} \le N \|D^2 v\|_{L_{p_0}\left((t_1-2,t_1)\times B_1\right)}.
\end{equation}
To estimate $I_2$, we see that
$\eta(s) = 0$ for any $s \in (-\infty, t-1)$ with $t \in (t_1-1,t_1)$.
Thus we have
$$
I_2(t,x) = -\eta(t) \int_{-\infty}^{t-1} (t-s)^{-\alpha-1} D^2 v(s,x) \, ds.
$$
Then,
\begin{align*}
|I_2(t,x)| &\le \int_{-\infty}^{t-1} |t-s|^{-\alpha-1} |D^2 v(s,x)| \, ds\\
&= \sum_{j=1}^\infty \int_{t-j-1}^{t-j} |t-s|^{-\alpha-1} |D^2 v(s,x)|\,ds\le \sum_{j=1}^\infty \int_{t-j-1}^{t-j } j^{-(\alpha+1)} |D^2 v(s,x)|\,ds.
\end{align*}
From this we have
\begin{equation*}
\|I_2\|_{L_{p_0}\left((t_1-1,t_1) \times B_1\right)}\le \sum_{j=1}^\infty j^{-(\alpha+1)} \left\| \int_{t-j-1}^{t-j} |D^2 v(s,x)| \, ds \right\|_{L_{p_0}\left((t_1-1,t_1) \times B_1\right)}.
\end{equation*}
Since $t_1 - 1 < t < t_1$,
$$
\int_{t-j-1}^{t-j} |D^2 v(s,x)|\, ds \leq \int_{t_1-j-2}^{t_1-j} |D^2 v(s,x)|\, ds.
$$
Hence, by H\"older's inequality,
\begin{equation*}
\|I_2\|_{L_{p_0}\left(Q_1(t_1,0)\right)}\le N\sum_{j=1}^\infty j^{-(\alpha+1)} \|D^2 v\|_{L_{p_0}\left((t_1-j-2,t_1-j)\times B_1\right)}.
\end{equation*}
Combining the above inequality, \eqref{eq0715_01}, and \eqref{eq0715_02}, we reach \eqref{eq0108_01} with $r=1$.

For the proof of \eqref{eq0108_02}, if $p_0 > d/2 + 1/\alpha$, we find $\tilde{p}_0$ such that $\tilde{p}_0 \leq p_0$ and $\tilde{p}_0 \in (d/2+1/\alpha, d+ 2/\alpha)$.
By \eqref{eq0109_01} and the Sobolev embeddings in \cite{MR3899965}, we have
\begin{equation*}
D^2(\eta v) \in C^{\sigma \alpha/2, \sigma}\left((t_1-1,t_1) \times B_{3/4}\right),
\end{equation*}
where $\sigma = 2 - (d+2/\alpha)/\tilde{p}_0 \in (0,1)$.
We then repeat the above steps from the inequalities in \eqref{eq0715_01} with $[D^2 v]_{C^{\sigma \alpha/2,\sigma}(Q_{1/2}(t_0,0))}$ in place of $\|D^2 v\|_{L_{p_1}\left(Q_{1/2}(t_0,0)\right)}$.

We now show that $v \in \bH_{p_1,0}^{\alpha,2}\left((0,t_0) \times B_{1/2}\right)$ when $p_1 < \infty$.
From \eqref{eq0109_02} and the equation \eqref{eq1.03} it follows that
\begin{equation*}
v, Dv, D^2 v, \partial_t^\alpha v \in L_{p_1}\left((0,t_0) \times B_{3/4}\right).
\end{equation*}
Note that as mentioned in  \cite[Remark 3.4]{MR3899965} this is not enough even to claim that $v$ belongs to $\bH_{p_1}^{\alpha,2}\left((0,t_0) \times B_{1/2}\right)$, which is a superset of $\bH_{p_1,0}^{\alpha,2}\left((0,t_0) \times B_{1/2}\right)$.
We take the mollification $v^{(\varepsilon)}$ of $v$ given in the proof of Proposition 3.2 in \cite{MR3899965}.
That is, we use
$$
v^{(\varepsilon)}(t,x) = \int_0^{t_0} \int_{B_1} \eta_\varepsilon(t-s,x-y) v(s,y) I_{0 < s < t_0} \, dy \, ds,
$$
where $\eta_\varepsilon(t,x) = \varepsilon^{-d-2/\alpha} \eta(t/\varepsilon^{2/\alpha},x/\varepsilon)$, $\eta(t,x)$ is an infinitely differentiable function defined in $\bR^{d+1}$ with compact support in $(0,1) \times B_1$ and $\int_{\bR^{d+1}} \eta \, dx \, dt = 1$.
By using the fact that $\eta(t,x) = 0$ for $t \leq 0$ and $v \in \bH_{p_0,0}^{\alpha,2}\left((0,t_0) \times B_1\right)$,
one can check that $v^{(\varepsilon)}(0,x) = 0$ and, for $(t,x) \in (0,t_0) \times B_{1/2}$,
$$
\partial_t^\alpha v^{(\varepsilon)}(t,x) = \int_0^{t_0} \int_{B_1} \eta_\varepsilon(t-s,x-y) \partial_t^\alpha v(s,y) \, dy \, ds,
$$
$$
D^m_x v^{(\varepsilon)}(t,x) = \int_0^{t_0} \int_{B_1} \eta_\varepsilon(t-s,x-y) D^m_x v(s,y) \, dy \, ds, \quad m=0,1,2.
$$
Then
$$
\|v^{(\varepsilon)} - v\|_{\bH_{p_1}^{\alpha,2}\left((0,t_0) \times B_{1/2}\right)} \to 0
$$
as $\varepsilon \to 0$.
This implies that $v \in \bH_{p_1,0}^{\alpha,2}\left((0,t_0) \times B_{1/2}\right)$. The lemma is proved.
\end{proof}

We need the following simple inequality.
\begin{lemma}
                    \label{lem2}
For any $\alpha>0$ and $k = 2,3,\ldots$, we have
$$
\sum_{j=1}^{k-1} j^{-(1+\alpha)}(k-j)^{-(1+\alpha)}\le N(\alpha) k^{-(1+\alpha)}.
$$
\end{lemma}
\begin{proof}
By using H\"older's inequality on $(1/j + 1/(k-j))^{1+\alpha}$, we have
\begin{align*}
&\sum_{j=1}^{k-1} k^{(1+\alpha)}j^{-(1+\alpha)}(k-j)^{-(1+\alpha)}
= \sum_{j=1}^{k-1} (1/j+1/(k-j))^{1+\alpha}\\
&\le 2^\alpha\sum_{j=1}^k \bigg(1/j^{1+\alpha}+1/(k-j)^{1+\alpha}\bigg)\le N(\alpha).
\end{align*}
The lemma is proved.
\end{proof}

\begin{proposition}
                            \label{prop1}
Let $1<p_0<p<\infty$, $t_0 \in (0,\infty)$, and $v\in \bH^{\alpha,2}_{p_0,0}((0,t_0)\times B_r)$ satisfy \eqref{eq1.03} in $(0,t_0) \times B_r$, $r > 0$. Then we have
\begin{align}
							\label{eq2.07}
\left( |D^2v|^{p} \right)_{Q_{r/2}(t_1,0)}^{1/p} \leq N
\sum_{j=1}^\infty j^{-(1+\alpha)} \left( |D^2v|^{p_0} \right)_{Q_{r}(t_1 -(j-1)r^{2/\alpha},0)}^{1/p_0}
\end{align}
for any $t_1 \leq t_0$,
where $N=N(d,\delta,\alpha,p,p_1)$. Furthermore,
\begin{align}
							\label{eq2.11}
[D^2 v]_{C^{\sigma \alpha/2,\sigma}(Q_{r/2}(t_1,0))} \leq N r^{-\sigma}
\sum_{j=1}^\infty j^{-(1+\alpha)} \left( |D^2v|^{p_0} \right)_{Q_r(t_1-(j-1)r^{2/\alpha},0)}^{1/p_0}
\end{align}
for any $t_1 \leq t_0$,
where $\sigma = \sigma(d,\alpha,p_0) \in (0,1)$.
\end{proposition}

\begin{proof}
Due to scaling, we consider only the case $r=1$.
By Lemma \ref{lem1}, one can find $p_1$ satisfying \eqref{eq0411_05} such that
\begin{equation}
							\label{eq2.41}
\left( |D^2v|^{p_1} \right)_{Q_{1/2}(t,0)}^{1/p_1} \leq N
\sum_{j=1}^\infty j^{-(1+\alpha)} \left( |D^2v|^{p_0} \right)_{Q_{1}(t-j+1,0)}^{1/p_0}
\end{equation}
for any $t\le t_0$.
If $p_1 \ge p$, we reach \eqref{eq2.07}.
Otherwise, by the same lemma, $v \in \bH_{p_1,0}^{\alpha,2}\left((0,t_0) \times B_{1/2}\right)$ and there exists $p_2$ satisfying
$$
p_2 > p_1 + \min\left\{\frac{2\alpha}{\alpha d + 2 - 2\alpha}, \alpha, \frac{2}{d} \right\}
$$
and
\begin{equation}
							\label{eq2.23}
\left( |D^2v|^{p_2} \right)_{Q_{1/4}(t_1,0)}^{1/p_2} \leq N
\sum_{j=1}^\infty j^{-(1+\alpha)} \left( |D^2v|^{p_1} \right)_{Q_{1/2}(t_1-(j-1)2^{-2/\alpha},0)}^{1/p_1}.
\end{equation}
Combining \eqref{eq2.41} and \eqref{eq2.23}, we get
\begin{align*}
\left( |D^2v|^{p_2} \right)_{Q_{1/4}(t_1,0)}^{1/p_2}
&\leq N \sum_{j=1}^\infty j^{-(1+\alpha)} \left( |D^2v|^{p_1} \right)_{Q_{1/2}(t_1-(j-1)2^{-2/\alpha},0)}^{1/p_1}\\
&\leq N \sum_{j=1}^\infty j^{-(1+\alpha)} \sum_{k=1}^\infty k^{-(1+\alpha)}\left( |D^2v|^{p_0} \right)_{Q_{1}(t_1-(j-1)2^{-2/\alpha}-k+1,0)}^{1/p_0}.
\end{align*}
Note that
\begin{align*}
&\sum_{j=1}^\infty j^{-(1+\alpha)} \sum_{k=1}^\infty k^{-(1+\alpha)}\left( |D^2v|^{p_0} \right)_{Q_1(t_1-(j-1)2^{-2/\alpha}-k+1,0)}^{1/p_0}\\
&= \sum_{m=1}^\infty \sum_{\substack{j \in \bN, \, j \geq 1 \\ m -1 \leq (j-1)2^{-2/\alpha} < m}} j^{-(1+\alpha)} \sum_{k=1}^\infty k^{-(1+\alpha)}\left( |D^2v|^{p_0} \right)_{Q_{1}(t_1-(j-1)2^{-2/\alpha}-k+1,0)}^{1/p_0}\\
&\leq \sum_{m=1}^\infty \sum_{\substack{j \in \bN, \, j \geq 1 \\ m -1 \leq (j-1)2^{-2/\alpha} < m}} m^{-(1+\alpha)} \sum_{k=1}^\infty k^{-(1+\alpha)}\left( |D^2v|^{p_0} \right)_{Q_{1}(t_1-(j-1)2^{-2/\alpha}-k+1,0)}^{1/p_0},
\end{align*}
where we used the inequalities
\begin{equation}
							\label{eq0113_01}
m-1 \leq (j-1) 2^{-2/\alpha} < m
\end{equation}
so that
$$
j^{-(1+\alpha)} \leq \left( (m-1) 2^{2/\alpha} + 1 \right)^{-(1+\alpha)} \leq m^{-(1+\alpha)}.
$$
Using \eqref{eq0113_01} again, we see that
\begin{align*}
&\left( |D^2v|^{p_0} \right)_{Q_{1}(t_1-(j-1)2^{-2/\alpha}-k+1,0)}^{1/p_0}\\
&\leq \left( |D^2v|^{p_0} \right)_{Q_{1}(t_1-m-k+1,0)}^{1/p_0} + \left( |D^2v|^{p_0} \right)_{Q_1(t_1-m-k+2,0)}^{1/p_0}.
\end{align*}
We then use the fact that for each $m$, the number of $j \in \bN$ such that
$$
m-1 \leq (j-1) 2^{-2/\alpha} < m
$$
is at most a fixed number determined by $\alpha$ to obtain that
\begin{align*}
&\left( |D^2v|^{p_2} \right)_{Q_{1/4}(t_1,0)}^{1/p_2} \leq N \sum_{m=1}^\infty m^{-(1+\alpha)} \sum_{k=1}^\infty k^{-(1+\alpha)}\left( |D^2v|^{p_0} \right)_{Q_{1}(t_1-m-k+1,0)}^{1/p_0}\\
&\qquad + N \sum_{m=1}^\infty m^{-(1+\alpha)} \sum_{k=1}^\infty k^{-(1+\alpha)} \left( |D^2v|^{p_0} \right)_{Q_1(t_1-m-k+2,0)}^{1/p_0},
\end{align*}
where for the first term on the right-hand side (in the same way for the second term), using Lemma \ref{lem2} we have
\begin{align*}
&\sum_{m=1}^\infty m^{-(1+\alpha)}\sum_{k=1}^\infty k^{-(1+\alpha)}\left( |D^2v|^{p_0} \right)_{Q_1(t_1-m-k+1,0)}^{1/p_0}\\
&= \sum_{m=1}^\infty m^{-(1+\alpha)} \sum_{k=m+1}^\infty (k-m)^{-(1+\alpha)} \left( |D^2v|^{p_0} \right)_{Q_1(t_1-k+1,0)}^{1/p_0}\\
&= \sum_{k=2}^\infty \left( |D^2v|^{p_0} \right)_{Q_1(t_1-k+1,0)}^{1/p_0} \sum_{m=1}^{k-1}m^{-(1+\alpha)} (k-m)^{-(1+\alpha)}\\
&\leq N \sum_{k=2}^\infty k^{-(1+\alpha)} \left( |D^2v|^{p_0} \right)_{Q_1(t_1-k+1,0)}^{1/p_0}.
\end{align*}
Combining the above inequalities, we arrive at
$$
\left( |D^2v|^{p_2} \right)_{Q_{1/4}(t_1,0)}^{1/p_2} \leq N \sum_{k=1}^\infty k^{-(1+\alpha)} \left( |D^2v|^{p_0} \right)_{Q_1(t_1-k+1,0)}^{1/p_0}.
$$
Repeating this procedure finite times and using a covering argument, we obtain \eqref{eq2.07}.

The inequality \eqref{eq2.11} follows directly from Lemma \ref{lem1} if $p_0 > d/2+1/\alpha$.
If $p_0 \leq d/2 + 1/\alpha$, we first have \eqref{eq2.07} with $r=1/2$ for a sufficiently large $p$ so that $p \in (d/2 + 1/\alpha,\infty)$ and $v \in \bH_{p,0}^{\alpha,2}\left((0,t_0) \times B_{1/4}\right)$.
Then by using \eqref{eq0108_02} with $r=1/4$ and the covering argument as above we arrive at \eqref{eq2.11} with $r=1$. The proposition is proved.
\end{proof}

\begin{remark}
The right-hand sides of \eqref{eq2.07} and \eqref{eq2.11} can be bounded by
$$
N(\cS\cM |D^2 v|^{p_0})^{1/p_0}(t_0,0),
$$
provided that $\cS\cM (|D^2 v|^{p_0})(t_0,0)$ is well defined.
Indeed, thanks to scaling it is enough to check this when $r=1$. By using H\"older's inequality (or the $l_1$-average is less than or equal to the $l_{p_0}$-average),
\begin{align}
                    \label{eq11.17}
&\sum_{j=1}^\infty j^{-(1+\alpha)} \left( |D^2v|^{p_0} \right)_{Q_{1}(t_0-j+1,0)}^{1/p_0}\nonumber\\
&\le  \sum_{k=0}^\infty
\sum_{j=2^k}^{2^{k+1}-1} 2^{-k(1+\alpha)} \left( |D^2v|^{p_0} \right)_{Q_{1}(t_0-j+1,0)}^{1/p_0}\nonumber\\
&\le  \sum_{k=0}^\infty
 2^{-k\alpha} \bigg[2^{-k}\sum_{j=2^k}^{2^{k+1}-1}\left( |D^2v|^{p_0} \right)_{Q_{1}(t_0-j+1,0)}\bigg]^{1/p_0}\nonumber\\
&=  \sum_{k=0}^\infty
 2^{-k\alpha} \left( |D^2v|^{p_0} \right)_{(t_0-2^{k+1}+1,t_0-2^{k}+1)\times B_1}^{1/p_0}\nonumber\\
& \le \sum_{k=0}^\infty  2^{-k\alpha}(\cS\cM |D^2 v|^{p_0})^{1/p_0}(t_0,0).
\end{align}
\end{remark}

\subsection{Estimates of \texorpdfstring{$w$}{w}}

Below we denote $\cH^{\alpha,1}_{2,0}\left((0,T) \times \bR^d\right)$ to be the set of functions which can be approximated by a sequence $\{u_n\} \subset C^\infty_0\left([0,T] \times \bR^d\right)$ with $u_n(0,\cdot) = 0$ in the norm
$$
\|u\|_{\cH^{\alpha,1}_2\left((0,T) \times \bR^d\right)} = \|u\|_{L_p\left((0,T) \times \bR^d\right)} + \|Du\|_{L_p\left((0,T) \times \bR^d\right)} + \|\partial_t^\alpha u\|_{\bH_p^{-1}\left((0,T) \times \bR^d\right)}.
$$
For details about $\cH^{\alpha,1}_{2,0}\left((0,T) \times \bR^d\right)$, see \cite{MR4030286}.

\begin{lemma}
                            \label{lem3}
Let $p_0\in (1,\infty)$, $t_0 \in (0,\infty)$, $f \in L_{p_0}\left((0,t_0) \times B_1\right)$, and $w\in \cH^{\alpha,1}_{2,0}((0,t_0)\times B_1)$ be a weak solution to \eqref{eq1.01} in $(0, t_0) \times B_1$ with the zero boundary condition on $\partial_p\left((0,t_0) \times B_1\right)$. Then we have
\begin{equation}
                                    \label{eq12.07}
\|w\|_{L_{p_0}((0,t_0)\times B_1)}\le N\|f\|_{L_{p_0}((0,t_0)\times B_1)},
\end{equation}
where $N=N(d,\delta,p_0)$ is independent of $t_0$.
\end{lemma}
\begin{proof}
We consider two cases.

{\em Case 1: $p_0\in [2,\infty)$.} For $\lambda\ge 1$, let $F(u)=|u|^{p_0}$ when $|u|\le \lambda$ and $F(u)=(p_0/2)\lambda^{p_0-2}|u|^2$ when $|u|>\lambda$. Note that $F$ is a $C^1$ convex function.
We multiply both sides of \eqref{eq1.01} by $-F'(w)$
and integrate in $(0,t_0)\times B_1$. By noting that because $F$ is convex the integral involving $\partial_t^\alpha w$ is nonnegative (see, for example, the proof of \cite[Lemma 4.1]{MR4030286}), we have
\begin{align*}
&\int_{(0,t_0)\times B_1} a_{ij}D_j wD_i w\big(p_0(p_0-1)|w|^{p_0-2}\chi_{|w|\le \lambda}+p_0\lambda^{p_0-2}\chi_{|w|>\lambda}\big) \, dx \, dt\\
&\le \int_{(0,t_0)\times B_1} |f|\big(p_0|w|^{p_0-1}\chi_{|w|\le \lambda}+p_0\lambda^{p_0-2}|w|\chi_{|w|>\lambda}\big)\, dx \, dt.
\end{align*}
By using the zero boundary condition and the Poincar\'e inequality to $G(w)$, where $G(u) = |u|^{p_0/2}$ when $|u| \leq \lambda$ and $G(u) = \lambda^{p_0/2 -1}|u|$ when $|u| > \lambda$, we then get
\begin{align*}
&\int_{(0,t_0) \times B_1} \left( |w|^{p_0} \chi_{|w| \leq \lambda} + \lambda^{p_0 - 2} |w|^2\chi_{|w| > \lambda} \right) \, dx \, dt
\\
&\quad= \int_{(0,t_0) \times B_1} |G(w)|^2 \, dx \, dt \leq N(d) \int_{(0,t_0) \times B_1} |D \left(G(w)\right)|^2 \, dx \,dt
\\
&\quad\leq N(d) \int_{(0,t_0) \times B_1} |Dw|^2 \left(\frac{p_0^2}{4} |w|^{p_0-2}  \chi_{|w| \leq \lambda} + \lambda^{p_0-2}\chi_{|w| > \lambda} \right) \, dx \, dt
\\
&\quad\leq N(d) \int_{(0,t_0) \times B_1} |Dw|^2 \left(p_0(p_0-1)|w|^{p_0-2}  \chi_{|w| \leq \lambda} + p_0 \lambda^{p_0-2}\chi_{|w| > \lambda} \right) \, dx \, dt
\\
&\quad\leq N(d,\delta)\int_{(0,t_0) \times B_1} |f|\left(p_0|w|^{p_0-1}\chi_{|w|\le \lambda}+p_0\lambda^{p_0-2}|w|\chi_{|w|>\lambda}\right) \, dx \,dt.
\end{align*}
From this and Young's inequality, we obtain that
\begin{align*}
&\int_{(0,t_0)\times B_1}|w\chi_{|w|\le \lambda}|^{p_0}+\lambda^{p_0-2}|w\chi_{|w| > \lambda}|^2\\
&\le N\int_{(0,t_0)\times B_1}|f|^{p_0} \chi_{|w|\le \lambda}+\lambda^{p_0-2}|f|^2\chi_{|w|> \lambda},
\end{align*}
where $N = N(d,\delta,p_0)$. This inequality with the fact that $f \in L_{p_0}\left((0,t_0) \times B_1\right)$ shows that
$$
\int_{(0,t_0) \times B_1} |w|^{p_0} \chi_{|w| \leq \lambda} \, dx \, dt \leq N \int_{(0,t_0) \times B_1} |f|^{p_0} \, dx \, dt.
$$
Taking $\lambda\to \infty$ and applying the monotone convergence theorem, we see that $w\in L_{p_0}((0,t_0)\times B_1)$ and \eqref{eq12.07} holds.

{\em Case 2: $p_0\in (1,2)$.} We use a duality argument. For any $g\in C^\infty((0,t_0)\times B_1)$, let $w_1$ be the unique $\cH^{\alpha,1}_{2,0}((0,t_0)\times B_1)$ of the equation
$$
-\partial^\alpha_t w_1+a^{ij}(t_0-t)D_{ij}w_1=g(t_0-t,x)
$$
in $((0,t_0)\times B_1)$ with the zero boundary condition on $\partial_p((0,t_0)\times B_1)$.
See \cite{MR2538276} for the existence and uniqueness of solutions.
From Case 1, we have
\begin{equation}
                            \label{eq12.38}
\|w_1\|_{L_{q_0}((0,t_0)\times B_1)}\le N\|g\|_{L_{q_0}((0,t_0)\times B_1)},
\end{equation}
where $q_0=p_0/(p_0-1)\in (2,\infty)$.
Denote $w_2(t,x)=w_1(t_0-t,x)$. Then by using integration by parts, we have
$$
\int_{(0,t_0)\times B_1}wg \, dx \, dt =\int_{(0,t_0)\times B_1}w_2 f \, dx \, dt.
$$
This together with \eqref{eq12.38} immediately gives \eqref{eq12.07}.

The lemma is proved.
\end{proof}

\begin{proposition}
                    \label{prop2}
Let $p_0\in (1,\infty)$, $f \in L_{p_0}\left( (0,t_0) \times B_1 \right)$, and $w\in \cH^{\alpha,1}_{2,0}((0,t_0)\times B_1)$ be a weak solution to \eqref{eq1.01} in $(0, t_0) \times B_1$ with the zero boundary condition on $\partial_p\left((0,t_0) \times B_1\right)$. Then we have
\begin{equation}
                                \label{eq3.02}
\left( |D^2 w|^{p_0} \right)_{Q_{1/2}(t_1,0)}^{1/p_0} \le \sum_{k=0}^\infty c_k(|f|^{p_0})^{1/p_0}_{(s_{k+1},s_k)\times B_{1}}
\end{equation}
for any $t_1 \leq t_0$, where $s_k = t_1 - 2^k + 1$ and
\begin{equation}
							\label{eq0212_03}
\sum_{k=0}^\infty c_k\le N=N(d,\delta,\alpha,p_0).
\end{equation}
\end{proposition}

\begin{proof}
Since the assumptions are satisfied if $t_0$ is replaced with $t_1$, we prove only the case $t_1 = t_0$.
By using the solvability in \cite{MR3899965} of equations in non-divergence form in $(0,t_0) \times \bR^d$, one can check that $w \in \bH_{p_0,0}^{\alpha,2}\left((0,t_0) \times B_{1-\varepsilon}\right)$ for any $\varepsilon \in (0,1)$.
By using a cutoff function as in the proof of Lemma \ref{lem1} (cf. \eqref{eq0715_01}), we have
\begin{equation}
                                \label{eq3.16}
\left( |D^2 w|^{p_0} \right)_{Q_{1/2}(t_0,0)}^{1/p_0}
\le N\left( |f|^{p_0} \right)_{Q_{1}(t_0,0)}^{1/p_0}
+N \sum_{k=0}^\infty 2^{-k\alpha} A_k,
\end{equation}
where
$$
A_k =(|w|^{p_0})^{1/p_0}_{(s_{k+1},s_k)\times B_{1}}, \quad s_k = t_0 - 2^k + 1.
$$
It remains to estimate $A_k$ for $k=0,1,\ldots$. To this end, we take cutoff functions $\eta_k\in C^\infty(\bR)$ such that $\eta_k=1$ for $t\ge s_{k+1}$, $\eta_k=0$ for $t\leq s_{k+2}$, and $\|\eta_k'\|_{L_\infty}\le 2^{-k}$. Then $w\eta_k$ satisfies
\begin{equation}
                            \label{eq3.54}
-\partial_t^\alpha (w\eta_k)+ a^{ij} D_{ij} (w\eta_k) = g_k
\end{equation}
in $(s_{k+2},t_0) \times B_1$
with the zero boundary condition on $\partial_p \left((s_{k+2},t_0) \times B_1\right)$, where $\partial_t^\alpha = \partial_t I_{s_{k+2}}^{1-\alpha}$ and
\begin{equation}
                        \label{eq4.29}
g_k=f\eta_k-\frac{\alpha}{\Gamma(1-\alpha)} \int_{\infty}^t (t-s)^{-\alpha-1}\left(\eta_k(t) - \eta_k(s)\right) w(s,x) \, ds
\end{equation}
in $(s_{k+2}, t_0) \times B_1$.
By Lemma \ref{lem3} applied to \eqref{eq3.54},
$$
\int_{(s_{k+2},s_k)\times B_1}|w\eta_k|^{p_0}\le N\int_{(s_{k+2},s_k)\times B_1}|g_k|^{p_0},
$$
which together with \eqref{eq4.29} further implies that

\begin{equation}
							\label{eq1216_01}
\int_{s_{k+1}}^{s_k} \int_{B_1} |w|^{p_0} \, dx \, dt \leq N \int_{s_{k+2}}^{s_k} \int_{B_1} |f|^{p_0} \, dx \, dt + N \|J\|_{L_{p_0}\left((s_{k+2},s_k) \times B_1\right)}^{p_0},
\end{equation}
where $N = N(d,\delta, \alpha, p_0)$
and
$$
J = \int_{-\infty}^t (t-s)^{-\alpha-1} \left(\eta_k(s) - \eta_k(t) \right) w(s,x) \, ds.
$$
Note that $\eta_k(s) - \eta_k(t) = 0$ for $s \in (s_{k+1},t)$ and $t \in (s_{k+1},s_k)$.
Thus,
\begin{align*}
J &= \int_{-\infty}^t (t-s)^{-\alpha-1}\left(\eta_k(s) - \eta_k(t) \right) w(s,x) \, \chi_{s \leq s_{k+1}} \, ds
\\
&= \int_{s_{k+3}}^t (t-s)^{-\alpha-1} \left(\eta_k(s) - \eta_k(t) \right) w(s,x)  \, \chi_{s \leq s_{k+1}} \, ds
\\
&+ \int_{-\infty}^{s_{k+3}} (t-s)^{-\alpha-1} \left(\eta_k(s) - \eta_k(t) \right) w(s,x) \, ds =: J_1 + J_2,
\end{align*}
where $t \in (s_{k+2},s_k)$ and $t-s \geq s_{k+2} - s_{k+3} = 2^{k+2}$ for $s \leq s_{k+3}$.
Since
$$
|\eta_k(t) - \eta_k(s)| \leq 2^{-k}|t-s|,
$$
it follows that
\begin{align*}
J_1 &\leq 2^{-k} \int_{s_{k+3}}^t (t-s)^{-\alpha}|w(s,x)| \, \chi_{s \leq s_{k+1}} \,d s
\\
&\leq 2^{-k} \int_0^{7 \cdot 2^k} s^{-\alpha} |w(t-s,x)| \, \chi_{s \geq t-s_{k+1}} \, ds.
\end{align*}
Using the Minkowski inequality, we have
\begin{align*}
\|J_1\|_{L_{p_0}\left((s_{k+2},s_k) \times B_1\right)} &\leq 2^{-k} \int_0^{7 \cdot 2^k} s^{-\alpha} \|w(\cdot -s,\cdot) \, \chi_{\cdot - s \leq s_{k+1}}\|_{L_{p_0}\left((s_{k+2},s_k) \times B_1\right)} \, ds
\\
&\leq N(\alpha) 2^{-\alpha k} \|w\|_{L_{p_0}\left((s_{k+4},s_{k+1}) \times B_1\right)}
\\
&\leq N 2^{-\alpha k} \sum_{j=1}^3\|w\|_{L_{p_0}\left((s_{k+j+1},s_{k+j}) \times B_1\right)}.
\end{align*}

From the fact that
$$
|\eta_k(s) - \eta_k(t)| = | \eta_k(t)| \leq 1
$$
for $s < s_{k+3}$, it follows that
$$
|J_2| \leq \int_{-\infty}^{s_{k+3}} (t-s)^{-\alpha-1} |w(s,x)| \, ds = \sum_{j=k+3}^\infty \int_{s_{j+1}}^{s_j} (t-s)^{-\alpha-1}|w(s,x)|\,ds.
$$
Since
$$
(t-s)^{-\alpha-1} \leq 2^{\alpha+1} 2^{-j(\alpha+1)}
$$
for $s \in (s_{j+1},s_j)$ and $t \in (s_{k+2},s_k)$ with $j\ge k+3$, we have
$$
|J_2| \leq 2^{\alpha+1} \sum_{j=k+3}^\infty 2^{-j(\alpha+1)} \int_{s_{j+1}}^{s_j} |w(s,x)| \, ds.
$$
Then by the Minkowski inequality and H\"older's inequality,
\begin{align*}
&\|J_2\|_{L_{p_0} \left((s_{k+2},s_k) \times B_1 \right)} \leq 2^{\alpha+1} \sum_{j=k+3}^\infty 2^{-j(\alpha + 1)} \int_{s_{j+1}}^{s_j} \|w(s,\cdot)\|_{L_{p_0} \left((s_{k+2},s_k) \times B_1 \right)} \, ds\\
&\quad \leq N(d) 2^{\alpha+1} \sum_{j=k+3}^\infty 2^{-\alpha j} (3\cdot 2^k)^{1/p_0} \left(\dashint_{\!s_{j+1}}^{\,\,\,s_j} \dashint_{B_1} |w(s,x)|^{p_0} \, ds \, dx \right)^{1/p_0}.
\end{align*}
Combining the estimates for $J_1$ and $J_2$ with \eqref{eq1216_01}, we arrive at
\begin{equation}
                            \label{eq4.25}
A_k\le N(|f|^{p_0})^{1/p_0}_{(s_{k+2},s_k)\times B_{1}}
+N_0\sum_{j=k+1}^\infty 2^{-\alpha j}A_j,
\end{equation}
where the constants $N$ and $N_0$ depend only on $d$, $\delta$, $\alpha$, and $p_0$.
Now we take a large $k_0$ such that $N_02^{-\alpha k_0}/(1-2^{-\alpha})\le 1/2$. Multiplying the above inequality by $2^{-\alpha k}$ and summing in $k=k_0,k_0+1,\ldots$, we get
\begin{align*}
\sum_{k=k_0}^\infty 2^{-\alpha k}A_k
&\le N\sum_{k=k_0}^\infty 2^{-\alpha k}(|f|^{p_0})^{1/p_0}_{(s_{k+2},s_k)\times B_{1}}+N_0\sum_{k=k_0}^\infty 2^{-\alpha k}\sum_{j=k+1}^\infty 2^{-\alpha j}A_j\\
&=N\sum_{k=k_0}^\infty 2^{-\alpha k}(|f|^{p_0})^{1/p_0}_{(s_{k+2},s_k)\times B_{1}}+N_0\sum_{j=k_0+1}^\infty \bigg(2^{-\alpha j}A_j\sum_{k=k_0}^{j-1} 2^{-\alpha k}\bigg)\\
&\le N\sum_{k=k_0}^\infty 2^{-\alpha k}(|f|^{p_0})^{1/p_0}_{(s_{k+2},s_k)}+\frac 12 \sum_{j=k_0+1}^\infty 2^{-\alpha j}A_j.
\end{align*}
Therefore, we have
$$
\sum_{k=k_0}^\infty 2^{-\alpha k}A_k \le N\sum_{k=k_0}^\infty 2^{-\alpha k}(|f|^{p_0})^{1/p_0}_{(s_{k+2},s_k)}.
$$
Finally, for $k=k_0-1,k_0-2,\ldots, 0$, by using \eqref{eq4.25} and induction, we get
$$
A_k \le N\sum_{j=0}^\infty c_j(|f|^{p_0})^{1/p_0}_{(s_{j+2},s_j)}.
$$
Combining the above inequalities with \eqref{eq3.16}, we reach \eqref{eq3.02} with different $c_k$'s. The proposition is proved.
\end{proof}

\subsection{Estimates of \texorpdfstring{$u$}{u}}

\begin{proposition}
							\label{prop0211_1}
Let $p_0 \in (1,\infty)$, $T \in (0,\infty)$, and $u \in \bH_{p_0,0,\operatorname{loc}}^{\alpha,2} \left((0,T) \times \bR^d\right)$ satisfy \eqref{eq0525_01} in $(0,T) \times \bR^d$.
Then, for any $(t_0,x_0) \in (0,T] \times \bR^d$, $r \in (0,\infty)$, and $\kappa \in (0,1/4)$, we have
\begin{equation*}
\begin{aligned}
&\left(|D^2u - (D^2u)_{Q_{\kappa r}(t_0,x_0)}| \right)_{Q_{\kappa r}(t_0,x_0)}
\leq N \kappa^{\sigma} (\cS\cM |D^2u|^{p_0})^{1/p_0}(t_0,x_0)
\\
& \quad + N \kappa^{-(d+2/\alpha)/p_0} \sum_{k=0}^\infty c_k \left(|f|^{p_0}\right)_{\left(t_0-(2^{k+2}-2)r^{2/\alpha},t_0\right) \times B_r(x_0)}^{1/p_0},
\end{aligned}
\end{equation*}
where $N = N(d,\delta,\alpha,p_0)$, $\sigma=\sigma(d,\alpha,p_0)$, and $\{c_k\}$ satisfies \eqref{eq0212_03}.
\end{proposition}

\begin{proof}
As in the parabolic case with local time derivative, thanks to translation and dilation, we assume that $x_0 = 1$ and $r=1$.
We further assume that $u$ and $f$ are sufficiently smooth so that, in particular, $f \in L_{p_0} \cap L_2\left((0,T) \times B_1 \right)$.
Indeed, since $u \in \bH_{p_0,0}^{\alpha,2}\left((0,T) \times B_R\right)$ for $R > 0$, there is a sequence $\{u_n\}$ satisfying $u_n \in C^\infty(\overline{(0,T) \times B_1})$, $u_n(0,x) = 0$, and \eqref{eq0316_01} as $n \to \infty$.
Then we prove the desired estimate, more precisely, \eqref{eq0316_02} with $u_n$ in place of $u$ and let $n \to \infty$.

Since $f \in L_{p_0} \cap L_2 \left((0,T) \times B_1 \right)$, one can use the results from \cite{MR2538276} to find a weak solution $w \in \cH_{2,0}^{\alpha,1}\left((0,t_0) \times B_1\right)$ satisfying \eqref{eq1.01} in $(0,t_0) \times B_1$ with the zero boundary condition on $\partial_p \left( (0,t_0) \times B_1 \right)$.
As in the proof of Proposition \ref{prop2}, we see that $w \in \bH_{p_0,0}^{\alpha,2}\left((0,t_0) \times B_{1-\varepsilon}\right)$ for any sufficiently small $\varepsilon >0$.
Set $v = u -w$, which belongs to $\bH_{p_0,0}^{\alpha,2}\left((0,t_0) \times B_{1-\varepsilon}\right)$ and satisfies \eqref{eq1.03} in $(0,t_0) \times B_{1-\varepsilon}$.
Note that because $\kappa < 1/4$,
\begin{equation}
							\label{eq0205_01}
\left(|D^2v - (D^2v)_{Q_\kappa(t_0,0)}|\right)_{Q_\kappa(t_0,0)} \leq N \kappa^\sigma [D^2v]_{C^{\sigma\alpha/2, \sigma}\left(Q_{1/4}(t_0,0)\right)}.
\end{equation}
By Proposition \ref{prop1} with $r = 1/2$ as well as the triangle inequality, we observe that
\begin{equation}
							\label{eq0205_02}
\begin{aligned}
&[D^2 v]_{C^{\sigma \alpha/2,\sigma}(Q_{1/4}(t_0,0))} \le N \sum_{j=1}^\infty j^{-(1+\alpha)} \left( |D^2v|^{p_0} \right)_{Q_{1/2}(t_0-(j-1)2^{-2/\alpha},0)}^{1/p_0}\\
& \qquad \le N
\sum_{j=1}^\infty j^{-(1+\alpha)} \left( |D^2u|^{p_0} \right)_{Q_{1/2}(t_0-(j-1)2^{-2/\alpha},0)}^{1/p_0}
\\
&\qquad \qquad +N\sum_{j=1}^\infty j^{-(1+\alpha)} \left( |D^2w|^{p_0} \right)_{Q_{1/2}(t_0-(j-1)2^{-2/\alpha},0)}^{1/p_0}.
\end{aligned}
\end{equation}
To estimate the summation involving $D^2 w$ above, we first note that by \eqref{eq3.02}, for $j=1,2,\ldots$,
$$
\left( |D^2w|^{p_0} \right)_{Q_{1/2}(t_0-(j-1)2^{-2/\alpha},0)}^{1/p_0} \leq \sum_{k=0}^\infty c_k \left(|f|^{p_0}\right)_{(s_{k+1}^j,s_k^j) \times B_1}^{1/p_0},
$$
where
$$
s_k^j = t_0-(j-1)2^{-2/\alpha}-2^k + 1.
$$
Then
\begin{align*}
&\sum_{j=1}^\infty j^{-(1+\alpha)} \left( |D^2w|^{p_0} \right)_{Q_{1/2}(t_0-(j-1)2^{-2/\alpha},0)}^{1/p_0}
\\
&\quad\leq \sum_{j=1}^\infty j^{-(1+\alpha)} \sum_{k=0}^\infty c_k \left(|f|^{p_0}\right)_{(s_{k+1}^j,s_k^j) \times B_1}^{1/p_0}
\\
&\quad= \sum_{m=1}^\infty \sum_{\substack{j \in \bN, \, j \geq 1\\ m-1 \leq (j-1)2^{-2/\alpha} < m}} j^{-(1+\alpha)} \sum_{k=0}^\infty c_k \left(|f|^{p_0}\right)_{(s_{k+1}^j,s_k^j) \times B_1}^{1/p_0}.
\end{align*}
For
$$
m-1 \leq (j-1)2^{-2/\alpha} < m,
$$
it holds that
$$
(s_{k+1}^j,s_k^j) \subset
\left(t_0 - 2^{k+1} + 1 - m, t_0 - m + 1\right).
$$
This shows that
$$
\left(|f|^{p_0}\right)_{(s_{k+1}^j,s_k^j) \times B_1}^{1/p_0} \leq 2^{1/p_0} \left(|f|^{p_0}\right)_{(s_{k+1}-m, t_0-m+1) \times B_1}^{1/p_0},
$$
where $s_k = t_0 - 2^k + 1$.
Thus,
\begin{align*}
&\sum_{j=1}^\infty j^{-(1+\alpha)} \left( |D^2w|^{p_0} \right)_{Q_{1/2}(t_0-(j-1)2^{-2/\alpha},0)}^{1/p_0}
\\
&\leq N \sum_{m=1}^\infty \sum_{\substack{j \in \bN, \, j \geq 1\\ m-1 \leq (j-1)2^{-2/\alpha} < m}} j^{-(1+\alpha)} \sum_{k=0}^\infty c_k \left(|f|^{p_0}\right)_{(s_{k+1}-m, t_0-m+1) \times B_1}^{1/p_0}
\\
&\leq N \sum_{m=1}^\infty m^{-(1+\alpha)}\sum_{k=0}^\infty c_k \left(|f|^{p_0}\right)_{(s_{k+1}-m, t_0-m+1) \times B_1}^{1/p_0}
\\
&\leq N \sum_{l=0}^\infty \sum_{m=2^l}^{2^{l+1}-1}2^{-l(1+\alpha)}\sum_{k=0}^\infty c_k \left(|f|^{p_0}\right)_{(s_{k+1}-m, t_0-m+1) \times B_1}^{1/p_0}
\\
&= N \sum_{l=0}^\infty 2^{-l(1+\alpha)} \sum_{k=0}^\infty c_k \sum_{m=2^l}^{2^{l+1}-1} \left(|f|^{p_0}\right)_{(s_{k+1}-m, t_0-m+1) \times B_1}^{1/p_0}
\\
&\leq N \sum_{l=0}^\infty 2^{-l \alpha} \sum_{k=0}^\infty c_k \left[\sum_{m=2^l}^{2^{l+1}-1} 2^{-l} \left(|f|^{p_0}\right)_{(s_{k+1}-m, t_0-m+1) \times B_1}\right]^{1/p_0},
\end{align*}
where in the last inequality we used H\"older's inequality.
We further observe that
\begin{align*}
&\sum_{m=2^l}^{2^{l+1}-1} 2^{-l} \left(|f|^{p_0}\right)_{(s_{k+1}-m, t_0-m+1) \times B_1}
\\
&= 2^{-l} 2^{-(k+1)}  \sum_{m=2^l}^{2^{l+1}-1} \int_{s_{k+1}-m}^{t_0-m+1} \dashint_{B_1} |f|^{p_0} \, dx \, dt
\\
&= 2^{-l} 2^{-(k+1)}  \sum_{m=2^l}^{2^{l+1}-1} \sum_{i=0}^{2^{k+1}-1} \left( |f|^{p_0} \right)_{Q_1(t_0-m-i+1,0)}
\\
&\leq 2^{-l-k} 2^{\min\{l,k\}} \sum_{m=2^l}^{2^{l+1}+2^{k+1}-2} \left( |f|^{p_0} \right)_{Q_1(t_0-m+1,0)}
\\
&
\leq
\left\{
\begin{aligned}
4 \left( |f|^{p_0} \right)_{(t_0-2^{l+2}+2,t_0) \times B_1} \quad &\text{for} \quad l \geq k,
\\
4 \left( |f|^{p_0} \right)_{(t_0-2^{k+2}+2,t_0) \times B_1} \quad &\text{for} \quad k > l.
\end{aligned}
\right.
\end{align*}
From the above two sets of estimates it follows that
$$
\sum_{j=1}^\infty j^{-(1+\alpha)}\left( |D^2 w|^{p_0} \right)_{Q_{1/2}(t_0-(j-1)2^{-2/\alpha},0)} \leq \sum_{k=0}^\infty c_k \left(|f|^{p_0}\right)_{(t_0-2^{k+2}+2,t_0) \times B_1}^{1/p_0},
$$
where $\{c_k\}$ is another sequence satisfying \eqref{eq0212_03}.
This combined with \eqref{eq0205_01} and \eqref{eq0205_02} shows that
\begin{align*}
\left(|D^2v - (D^2v)_{Q_\kappa(t_0,0)}|\right)_{Q_\kappa(t_0,0)} &\leq N \kappa^\sigma \sum_{j=1}^\infty j^{-(1+\alpha)} \left( |D^2u|^{p_0} \right)_{Q_{1/2}(t_0-(j-1)2^{-2/\alpha},0)}^{1/p_0}
\\
&\quad + N \kappa^\sigma \sum_{k=0}^\infty c_k \left(|f|^{p_0}\right)_{(t_0-2^{k+2}+2,t_0) \times B_1}^{1/p_0}.
\end{align*}
We then have
\begin{equation}
							\label{eq0316_02}
\begin{aligned}
&\left(|D^2 u - (D^2 u)_{Q_\kappa(t_0,0)}|\right)_{Q_\kappa(t_0,0)}
\\
&\leq \left(|D^2v - (D^2v)_{Q_\kappa(t_0,0)}|\right)_{Q_\kappa(t_0,0)} + N \kappa^{-(d+2/\alpha)/p_0} \left(|D^2w|\right)_{Q_{1/2}(t_0,0)}
\\
&\leq \kappa^\sigma \sum_{j=1}^\infty j^{-(1+\alpha)} \left( |D^2u|^{p_0} \right)_{Q_{1/2}(t_0-(j-1)2^{-2/\alpha},0)}^{1/p_0} \\
&\quad + N \kappa^{-(d+2/\alpha)/p_0}  \sum_{k=0}^\infty c_k \left(|f|^{p_0}\right)_{(t_0-2^{k+2}+2,t_0) \times B_1}^{1/p_0},
\end{aligned}
\end{equation}
where $N = N(d,\delta,\alpha,p_0)$ and we used in the last inequality  \eqref{eq3.02} along with the observation that
$$
(|f|^{p_0})^{1/p_0}_{(t_0 - 2^{k+1}+1, t_0 - 2^k+1)\times B_1} \leq  N(p_0) (|f|^{p_0})^{1/p_0}_{(t_0-2^{k+2}+2,t_0)\times B_1}.
$$
Finally, we use \eqref{eq11.17} with a scaling and $u$ in place of $v$ to get
$$
\sum_{j=1}^\infty j^{-(1+\alpha)} \left( |D^2u|^{p_0} \right)_{Q_{1/2}(t_0-(j-1)2^{-2/\alpha},0)}^{1/p_0} \leq N (\cS\cM |D^2 u|^{p_0})^{1/p_0}(t_0,0),
$$
where $N = N(\alpha)$.
The proposition is proved.
\end{proof}

\section{Weight mixed-norm estimates}
                    \label{sec5}

In this section we prove Theorem \ref{thm0210_1}.

\begin{lemma}
							\label{lem0218_1}
Let $\alpha \in (0,1]$, $T \in (0,\infty)$, $p,q \in (1,\infty)$, $K_1 \in [1,\infty)$, $w = w_1(t) w_2(x)$, where
$$
w_1(t) \in A_p(\bR,dt), \quad w_2(x) \in A_q(\bR^d, dx), \quad [w_1]_{A_p} \leq K_1, \quad [w_2]_{A_q} \leq K_1.
$$
There exist $p_0 = p_0(d,p,q,K_1) \in (1,\infty)$ and $\mu = \mu(d,p,q,K_1) \in (1,\infty)$ such that
$$
p_0 < p_0 \mu < \min\{p,q\}
$$
and the following holds.

If $u \in \bH_{p, q, w, 0}^{\alpha,2}\left((0,T) \times \bR^d\right)$ has compact support in $[0,T] \times B_{R_0}$ and satisfies
\begin{equation}
							\label{eq0218_02}
-\partial_t^\alpha u + a^{ij}(t,x) D_{ij} u = f
\end{equation}
in $(0,T) \times \bR^d$, where the coefficients $a^{ij}(t,x)$ satisfy Assumption \ref{assum0210_1} ($\gamma_0$),
then for any $(t_0,x_0) \in (0,T] \times \bR^d$, $r \in (0,\infty)$, $\kappa \in (0,1/4)$,
we have
\begin{equation}
							\label{eq0219_01}
\begin{aligned}
&\left(|D^2u - (D^2u)_{Q_{\kappa r}(t_0,x_0)}| \right)_{Q_{\kappa r}(t_0,x_0)}
\leq N \kappa^{\sigma} (\cS\cM |D^2u|^{p_0})^{1/p_0}(t_0,x_0)
\\
&\qquad\qquad + N \kappa^{-(d+2/\alpha)/p_0} \gamma_0^{1/\nu p_0} (\cS\cM |D^2u|^{\mu p_0})^{1/\mu p_0}(t_0,x_0)
\\
&\qquad\qquad + N \kappa^{-(d+2/\alpha)/p_0} \left(\cS\cM |f|^{p_0}\right)^{1/p_0}(t_0,x_0),
\end{aligned}
\end{equation}
where $\nu = \mu/(\mu-1)$, $\gamma = \gamma(d,\alpha,p_0)$, and $N = N(d,\delta, \alpha, p,q,K_1)$.
The functions $u$ and $f$ are defined to be zero whenever $t \leq 0$.
\end{lemma}

\begin{proof}
For the given $w_1 \in A_p(\bR, d t)$ and $w_2 \in A_q(\bR^d, dx)$, using the reverse H\"{o}lder's inequality for $A_p$ weights, we find
$$
\sigma_1 = \sigma_1(d,p,K_1), \quad \sigma_2 = \sigma_2(d,q,K_1)
$$
such that $p - \sigma_1 >1$, $q - \sigma_2 > 1$, and
$$
w_1 \in A_{p-\sigma_1}(\bR, d t), \quad w_2 \in A_{q-\sigma_2}(\bR^d, dx).
$$
Set $p_0, \mu \in (1,\infty)$ so that
$$
p_0 \mu = \min \left\{ \frac{p}{p-\sigma_1}, \frac{q}{q-\sigma_2} \right\} > 1.
$$
Note that
$$
w_1 \in A_{p-\sigma_1} \subset A_{\frac{p}{p_0\mu}} \subset A_{\frac{p}{p_0}}(\bR, dt),
$$
$$
w_2 \in A_{q-\sigma_2} \subset A_{\frac{q}{p_0\mu}} \subset A_{\frac{q}{p_0}}(\bR^d, dx).
$$
From these inclusions and the fact that $u \in \bH_{p,q,w,0}^{\alpha,2}\left((0,T) \times \bR^d\right)$ it follows that (see the proof of \cite[Lemma 5.10]{MR3812104})
$$
u \in \bH_{p_0\mu, 0, \operatorname{loc}}^{\alpha,2}\left((0,T) \times \bR^d\right).
$$
In particular, if $\alpha = 1$, by extending $u$ as zero for $t < 0$, we see that
$$
u \in W_{p_0\mu,\operatorname{loc}}^{1,2}\left((-\infty,T) \times \bR^d\right).
$$
Set
$$
\bar{a}^{ij}(t) =
\left\{
\begin{aligned}
\dashint_{B_r(x_0)} a^{ij}(t,y) \, dy \quad &\text{if} \quad r \leq R_0,
\\
\dashint_{B_{R_0}} a^{ij}(t,y) \, dy \quad &\text{if} \quad r > R_0,
\end{aligned}
\right.
$$
and write
$$
-\partial_t^\alpha u + \bar{a}^{ij}(t) D_{ij} u = \tilde{f},
$$
where
$$
 \tilde{f} = f + \left(\bar{a}^{ij}(t) - a^{ij}(t,x) \right) D_{ij} u.
$$
Then by Propositions \ref{prop0204_1} and \ref{prop0211_1} it follows that
\begin{equation}
							\label{eq0218_01}
\begin{aligned}
\left(|D^2u - (D^2u)_{Q_{\kappa r}(t_0,x_0)}| \right)_{Q_{\kappa r}(t_0,x_0)} &\leq N \kappa^{\sigma} (\cS\cM |D^2u|^{p_0})^{1/p_0}(t_0,x_0)
\\
+ N \kappa^{-(d+2/\alpha)/p_0} \sum_{k=0}^\infty c_k & \left(|\tilde{f}|^{p_0}\right)_{\left(t_0-(2^{k+2}-2)r^{2/\alpha},t_0\right) \times B_r(x_0)}^{1/p_0},
\end{aligned}
\end{equation}
where $\{c_k\}$ satisfies \eqref{eq0212_04} or
\eqref{eq0212_03} and $N = N(d, \delta,\alpha,p_0) = N(d,\delta,\alpha,p,q,K_1)$.
Using the fact that $u$ has compact support in $[0,T] \times B_{R_0}$ and H\"older's inequality, we write
$$
\left(|\tilde{f}|^{p_0}\right)_{\left(t_0-(2^{k+2}-2)r^{2/\alpha},t_0\right) \times B_r(x_0)}^{1/p_0} \leq \left(|f|^{p_0}\right)_{\left(t_0-(2^{k+2}-2)r^{2/\alpha},t_0\right) \times B_r(x_0)}^{1/p_0}
$$
$$
+ \left(|\bar{a}^{ij} - a^{ij}|^{\nu p_0} \chi_{B_{R_0}} \right)_{\left(t_0-(2^{k+2}-2)r^{2/\alpha},t_0\right) \times B_r(x_0)}^{1/\nu p_0}
$$
$$
\cdot \left(|D^2u|^{\mu p_0}\right)_{\left(t_0-(2^{k+2}-2)r^{2/\alpha},t_0\right) \times B_r(x_0)}^{1/\mu p_0},
$$
where by Assumption \ref{assum0210_1} and the boundedness of $a^{ij}$ (also see Remark 2.3 in \cite{MR3899965}), for $r \leq R_0$,
\begin{align*}
&\left(|\bar{a}^{ij} - a^{ij}|^{\nu p_0}\chi_{B_{R_0}}\right)_{\left(t_0-(2^{k+2}-2)r^{2/\alpha},t_0\right) \times B_r(x_0)}\\
&\leq \left(|\bar{a}^{ij} - a^{ij}|^{\nu p_0}\right)_{\left(t_0-(2^{k+2}-2)r^{2/\alpha},t_0\right) \times B_r(x_0)}  \leq N \gamma_0,
\end{align*}
and, for $r > R_0$,
\begin{align*}
&\left(|\bar{a}^{ij} - a^{ij}|^{\nu p_0}\chi_{B_{R_0}}\right)_{\left(t_0-(2^{k+2}-2)r^{2/\alpha},t_0\right) \times B_r(x_0)}\\
&\leq \left(|\bar{a}^{ij} - a^{ij}|^{\nu p_0}\right)_{\left(t_0-(2^{k+2}-2)r^{2/\alpha},t_0\right) \times B_{R_0}} \leq N \gamma_0.
\end{align*}
Thus,
\begin{align*}
&\sum_{k=0}^\infty c_k \left(|\tilde{f}|^{p_0}\right)_{\left(t_0-(2^{k+2}-2)r^{2/\alpha},t_0\right) \times B_r(x_0)}^{1/p_0}
\leq \sum_{k=0}^\infty c_k\left(|f|^{p_0}\right)_{\left(t_0-(2^{k+2}-2)r^{2/\alpha},t_0\right) \times B_r(x_0)}^{1/p_0}\\
&\quad + N \gamma_0^{1/\nu p_0} \sum_{k=0}^\infty c_k \left(|D^2u|^{\mu p_0}\right)_{\left(t_0-(2^{k+2}-2)r^{2/\alpha},t_0\right) \times B_r(x_0)}^{1/\mu p_0}\\
&\leq N \left(\cS\cM |f|^{p_0}\right)^{1/p_0}(t_0,x_0) + N \gamma_0^{1/\nu p_0} \left(\cS\cM |D^2 u|^{\mu p_0}\right)^{1/\mu p_0}(t_0,x_0).
\end{align*}
This inequality together with \eqref{eq0218_01} proves \eqref{eq0219_01}.
The lemma is proved.
\end{proof}

For each integer $n \in \bZ$, find an integer $k(n)$ such that
$$
k(n)  \leq \frac{2}{\alpha}n < k(n) + 1.
$$
Note that
$k(n+1) - k(n)$ is a non-negative integer and
$$
\frac{1}{2^{k(n)+1}} < \frac{1}{2^{2n/\alpha}} \leq \frac{1}{2^{k(n)}}.
$$
Let
$$
\bC_n := \left\{ Q^n_{\vec{i}} = Q^n_{(i_0,i_1,\ldots,i_d)}: \vec{i} = (i_0, i_1, \ldots, i_d) \in \bZ^{d+1}\right\},
$$
where $n \in \bZ$ and
$$
Q^n_{\vec{i}}
= \left[ \frac{i_0}{2^{k(n)}}+T, \frac{i_0+1}{2^{k(n)}}+T\right) \times \left[ \frac{i_1}{2^n}, \frac{i_1+1}{2^n}\right) \times \cdots \times \left[ \frac{i_d}{2^n}, \frac{i_d+1}{2^n}\right).
$$
Note that $\{\bC_n\}_{n \in \bZ}$ is a collection of partitions of $\bR^{d+1}$ satisfying \cite[Theorem 2.1]{MR3812104} with respect to the parabolic distance
$$
|(t,x)-(s,y)| := \max\{ |x-y|, |t-s|^{\alpha/2}\}.
$$
In particular, for each $n \in \bZ$, $Q^n_{\vec{i}}$ belongs to either $(-\infty,T) \times \bR^d$ or $[T, \infty) \times \bR^d$.
Also note that, for each $Q_{\vec{i}}^n$, where $Q_{\vec{i}}^n \cap \left((-\infty,T) \times \bR^d\right) \neq \emptyset$, there exists $Q_r(t_0,x_0)$ such that $t_0 \in (-\infty,T]$ and
$$
Q_{\vec{i}}^n \subset Q_r(t_0,x_0), \quad |Q_r(t_0,x_0)| \leq N |Q_{\vec{i}}^n|,
$$
where $N  = N(d,\alpha)$.
Indeed, we can take, for example,
$$
t_0 = \frac{i_0+1}{2^{k(n)}}+T, \quad x_0=({x_0}_1, \ldots, {x_0}_d), \quad {x_0}_j = \frac{2i_l+1}{2^{n+1}},
$$
and
$$
r = \max\{\sqrt{d}/2,2^{\alpha/2}\}2^{-n}.
$$
Denote the dyadic sharp function of $g$ by
$$
g^{\#}_{\operatorname{dy}}(t,x) = \sup_{n < \infty} \dashint_{Q^n_{\vec{i}} \ni (t,x)} \left| g(s,y) - g_{|n}(t,x)\right| \, dy \, ds,
$$
where
$$
g_{|n}(t,x) = \dashint_{Q_{\vec{i}}^n} g(s,y) \, dy \, ds, \quad (t,x) \in Q_{\vec{i}}^n.
$$

\begin{theorem}[Maximal function theorem for strong maximal functions]
							\label{thm0225_1}
Let $p, q \in (1,\infty)$, $K_1 \in [1, \infty)$, $w_1(t) \in A_p(\bR,dt)$, $w_2(x) \in A_q(\bR, dx)$, $[w_1]_{A_p} \leq K_1$, $[w_2]_{A_q} \leq K_1$, and $w(t,x) = w_1(t)w_2(x)$.
Then, for any $f \in L_{p,q,w}(\bR \times \bR^d)$, we have
$$
\|\cS\cM f\|_{L_{p,q,w}(\bR \times \bR^d)} \leq N \|f\|_{L_{p,q,w}(\bR \times \bR^d)},
$$
where $N = N(d,p,q,K_1) > 0$.
\end{theorem}

\begin{proof}
When $p=q$, this follows from \cite[Theorem 1.1]{B86}. The general case is a consequence of the extrapolation theorem of Rubio de Francia \cite{MR745140}. See also \cite[Theorem 2.5]{MR3812104}.
\end{proof}

\begin{lemma}
							\label{lem0221_1}
Let $\alpha \in (0,1]$, $T \in (0,\infty)$, $p, q \in (1,\infty)$, $K_1 \in [1,\infty)$, $w = w_1(t) w_2(x)$, where
$$
w_1(t) \in A_p(\bR,dt), \quad w_2(x) \in A_q(\bR^d, dx), \quad [w_1]_{A_p} \leq K_1, \quad [w_2]_{A_q} \leq K_1.
$$
There exists $\gamma_0 = \gamma_0(d,\delta,\alpha,p,q,K_1) > 0$ such that, under Assumption \ref{assum0210_1} ($\gamma_0$),
for any $u \in \bH_{p, q, w, 0}^{\alpha,2}\left((0,T) \times \bR^d\right)$ with compact support in $[0,T] \times B_{R_0}$ satisfying \eqref{eq0218_02} in $(0,T) \times \bR^d$, we have
\begin{equation}
							\label{eq0220_01}
\|\partial_t^\alpha u\|_{L_{p,q,w}\left((0,T) \times \bR^d\right)} + \|D^2u\|_{L_{p,q,w}\left((0,T) \times \bR^d\right)} \leq N \|f\|_{L_{p,q,w}\left((0,T) \times \bR^d\right)},
\end{equation}
where $N = N(d,\delta,\alpha,p,q,K_1)$.
\end{lemma}

\begin{proof}
By using the partitions $\bC_n$ and the dyadic sharp function introduced above, from \eqref{eq0219_01} we obtain that
\begin{align*}
(D^2u)_{\operatorname{dy}}^\#(t_0,x_0) &\leq N \kappa^\sigma (\cS\cM |D^2u|^{p_0})^{1/p_0}(t_0,x_0)
\\
&+ N \kappa^{-(d+2/\alpha)/p_0} \gamma_0^{1/\nu p_0} (\cS\cM |D^2u|^{\mu p_0})^{1/\mu p_0}(t_0,x_0)
\\
&+ N \kappa^{-(d+2/\alpha)/p_0} \left(\cS\cM |f|^{p_0}\right)^{1/p_0}(t_0,x_0)
\end{align*}
for any $(t_0,x_0) \in \bR \times \bR^d$, provided that $D^2u$ is defined to be zero on $(T,\infty) \times \bR^d$.
Indeed, for $(t_0,x_0) \in (T,\infty) \times \bR^d$, we see that
$$
(D^2u)_{\operatorname{dy}}^{\#}(t_0,x_0) = 0
$$
by the choice of the partitions.
Then by the sharp function theorem (see \cite[Corollary 2.7]{MR3812104}) and Theorem \ref{thm0225_1} we get
$$
\|D^2u\|_{p,q,w} \leq N \left(\kappa^\sigma + \kappa^{-(d+2/\alpha)/p_0} \gamma_0^{1/\nu p_0}\right) \|D^2u\|_{p,q,w}
$$
$$
+ N \kappa^{-(d+2/\alpha)/p_0} \|f\|_{p,q,w},
$$
where $\|\cdot\|_{p,q,w} = \|\cdot\|_{L_{p,q,w}\left((0,T) \times \bR^d\right)}$ and $N = N(d,\delta,\alpha,p,q,K_1)$.
Now by first taking a sufficiently small $\kappa < 1/4$, then taking small $\gamma_0$ so that
$$
\kappa^\sigma + \kappa^{-(d+2/\alpha)/p_0} \gamma_0^{1/\nu p_0} < \frac{1}{2N},
$$
we arrive at
$$
\|D^2u\|_{L_{p,q,w}\left((0,T) \times \bR^d\right)} \leq N \|f\|_{L_{p,q,w}\left((0,T) \times \bR^d\right)}.
$$
Then, using this estimate and the equation, we obtain \eqref{eq0220_01}.
The lemma is proved.
\end{proof}

\begin{corollary}
							\label{cor0221_01}
Let $\alpha \in (0,1]$, $T \in (0,\infty)$, $p \in (1,\infty)$, $K_1 \in [1,\infty)$, $w = w_1(t) w_2(x)$, where
$$
w_1(t) \in A_p(\bR,dt), \quad w_2(x) \in A_q(\bR^d, dx), \quad [w_1]_{A_p} \leq K_1, \quad [w_2]_{A_q} \leq K_1.
$$
There exists $\gamma_0 = \gamma_0(d,\delta,\alpha,p,q,K_1) > 0$ such that, under Assumption \ref{assum0210_1} ($\gamma_0$),
for any $u \in \bH_{p, q, w, 0}^{\alpha,2}\left((0,T) \times \bR^d\right)$ satisfying \eqref{eq0210_02} in $(0,T) \times \bR^d$, we have
\begin{equation}
							\label{eq0220_02}
\|\partial_t^\alpha u\|_{p,q,w} + \|D^2u\|_{p,q,w} \leq N_0 \|f\|_{p,q,w} + N_1\|u\|_{p,q,w},
\end{equation}
where $N_0 = N_0(d,\delta,\alpha,p,q,K_1)$, $N_1 = N_1(d, \delta,\alpha,p,q, K_1,K_0,R_0)$, and $\|\cdot\|_{p,q,w} = \|\cdot\|_{L_{p,q,w}\left((0,T) \times \bR^d\right)}$.
\end{corollary}

\begin{proof}
We first consider the case $p=q$.
Write
$$
-\partial_t^\alpha u + a^{ij}D_{ij}u = f - b^i D_i u - cu.
$$
Then using Lemma \ref{lem0221_1} with $p=q$, and using the partition of unity with respect to the spatial variables, we have
$$
\|\partial_t^\alpha u\|_{p,w} + \|D^2 u\|_{p,w} \leq N_0\|f\|_{p,w} + N_1\|Du\|_{p,w} + N_1\|u\|_{p,w},
$$
where $N_0 = N_0(d,\delta,\alpha,p,q,K_1)$, $N_1 = N_1(d, \delta,\alpha,p,q,K_1,K_0,R_0)$, and $\|\cdot\|_{p,w}=\|\cdot\|_{L_{p,w}\left((0,T) \times \bR^d\right)}$.
Then we use an interpolation inequality (see \cite[Lemma 3.5 (iii)]{arXiv:1806.00077}) to derive \eqref{eq0220_02} for $p=q$.

For $p \neq q$, we use the extrapolation theorem. See \cite[Theorem 2.5]{MR3812104}.
\end{proof}

To estimate $\|u\|_{L_{p,q,w}\left((0,T) \times \bR^d\right)}$ on the right-hand side of \eqref{eq0220_02}, we need the following observation.

\begin{lemma}
                \label{lem0226_1}
Let $\alpha \in (0,1)$, $p, q\in (1,\infty)$ and $K_1 \in [1,\infty)$, $w(t,x) = w_1(t)w_2(x)$, where
$$
w_1\in A_p(\bR, dt), \quad w_2 \in A_p(\bR^d,dx), \quad [w_1]_{A_p} \leq K_1, \quad
[w_2]_{A_q} \leq K_1.
$$
\begin{enumerate}
\item For any $f\in L_{p,w_1}((0,T))$, we have
\begin{equation}
                \label{eq9.09}
\|I^\alpha f\|_{L_{p,w_1}((0,T))}\le NT^\alpha \|f\|_{L_{p,w_1}((0,T))},
\end{equation}
where $N>0$ depends only on $\alpha$, $p$, and $K_1$.

\item For any $g\in L_{p,q,w}((0,T)\times \bR^d)$, we have
\begin{equation}
                \label{eq9.25}
\|I^\alpha g\|_{L_{p,q,w}((0,T)\times \bR^d)}\le NT^\alpha \|g\|_{L_{p,q,w}((0,T)\times \bR^d)},
\end{equation}
where $N>0$ depends only on $\alpha$, $p$, $q$, $K_1$.
\end{enumerate}
\end{lemma}

\begin{proof}
By scaling, we may assume that $T=1$.
Indeed, when scaling, recall that $[w_1(T t)]_{A_p} = [w_1(t)]_{A_p}$.
We extend $f$ to be zero when $t\notin (0,T)$.
Set
$$
F(s) = \int_0^s |f(t-r)|\,dr,
$$
which is absolutely continuous.
Then, for $t \in [0,1]$,
$$
|I^\alpha f(t)| \leq \frac{1}{\Gamma(\alpha)} \int_0^1 s^{\alpha-1} |f(t-s)| \, ds = \frac{1}{\Gamma(\alpha)} \int_0^1 s^{\alpha-1} F'(s) \, ds,
$$
where
\begin{align*}
&\int_0^1 s^{\alpha-1} F'(s) \, ds = \lim_{\varepsilon \searrow 0}\int_\varepsilon^1 s^{\alpha-1} F'(s) \, ds\\
&= \lim_{\varepsilon \searrow 0} \left( F(1) - \varepsilon^{\alpha-1}F(\varepsilon) + (1-\alpha) \int_\varepsilon^1 s^{\alpha-2} F(s) \, ds \right).
\end{align*}
Since
$$
s^{-1}F(s) \leq 2 \cM f(t)
$$
for any $s \in (0,1]$,
where
$$
\cM f(t) = \frac{1}{2r}\int_{t-r}^{t+r} |f(s)| \, ds,
$$
we see that
$$
\int_0^1 s^{\alpha-1} F'(s) \, ds \leq N \cM f(t), \quad \text{i.e.}, \quad
|I^\alpha f(t)| \leq N \cM f(t),
$$
where $N=N(\alpha)$.
Then \eqref{eq9.09} follows from the Hardy--Littlewood maximal function theorem with $A_p$ weights. Finally, we get \eqref{eq9.25} by using the Fubini theorem and \eqref{eq9.09} when $p=q$, and the general case by using the extrapolation theorem. See \cite[Theorem 2.5]{MR3812104}.
The lemma is proved.
\end{proof}

\begin{lemma}
							\label{lem0226_2}
Let $\alpha \in (0,1]$, $T \in (0,\infty)$, $p \in (1,\infty)$, $K_1 \in [1,\infty)$, $w = w_1(t)w_2(x)$, where
$$
w_1(t) \in A_p(\bR,dt), \quad w_2(x) \in A_p(\bR^d,dx), \quad [w_1]_{A_p} \leq K_1, \quad [w_2]_{A_p} \leq K_2.
$$
For $u \in \bH_{p,q,w,0}^{\alpha,2}\left((0,T) \times \bR^d\right)$, we have
$$
\|u\|_{p,q,w} \leq N T^\alpha\|\partial_t^\alpha u\|_{p,q,w},
$$
where $N = N(\alpha,p,q,K_1)$ and $\|\cdot\|_{p,q,w} = \|\cdot\|_{L_{p,q,w}\left((0,T) \times \bR^d\right)}$.
\end{lemma}

\begin{proof}
We first consider the case $\alpha \in (0,1)$.
Since $u \in \bH_{p,q,w,0}^{\alpha,2}\left((0,T) \times \bR^d\right)$, we further assume that $u \in C^\infty_0\left([0,T] \times \bR^d\right)$ with $u(0,\cdot) = 0$.
Then by \cite[Lemma A.4]{MR3899965} and Lemma \ref{lem0226_1}
$$
\|u\|_{p,q,w} = \|I^\alpha \partial_t^\alpha u\|_{p,q,w} \leq N T^\alpha \|\partial_t^\alpha u\|_{p,q,w},
$$
where $N=N(\alpha,p,q,K_1)$.

For $\alpha = 1$,
since
$$
|u(t,x)| \leq \int_0^t |u_t(s,x)| \, ds \leq 2 t \cM u_t(t,x) \leq 2 T\cM u_t(t,x),
$$
the desired inequality follows as in the proof of Lemma \ref{lem0226_1}.
\end{proof}

We are now ready to present the proof of Theorem \ref{thm0210_1}.

\begin{proof}[Proof of Theorem \ref{thm0210_1}]
We first prove the estimate \eqref{eq0226_01}.
We only consider $\alpha \in (0,1)$ because the case $\alpha = 1$ is simpler.
By Corollary \ref{cor0221_01} and the interpolation inequality used in the proof of Corollary \ref{cor0221_01}, it suffices to show that
\begin{equation}
							\label{eq0226_02}
\|u\|_{p,q,w} \leq N \|f\|_{p,q,w}.
\end{equation}
By extending $u$ and $f$ as zero for $t < 0$, we observe (see \cite[Lemma 3.5]{MR3899965}) that
$$
-\partial_t^\alpha u + a^{ij}(t,x)D_{ij} u + b^i D_iu + cu = f
$$
in $(S,T) \times \bR^d$ for any $S \leq 0$, where
$$
\partial_t^\alpha u = \partial_t I^{1-\alpha}_S u.
$$
Take a positive integer $m$ to be specified below and set
$$
s_j = \frac{jT}{m}, \quad j = -1, 0, 1, 2,\ldots, m.
$$
We then take cutoff functions $\eta_j \in C^\infty(\bR)$, $j=0,1,2,\ldots,m-1$, such that
$$
\eta_j = \left\{
\begin{aligned}
1 \quad &\text{for} \quad t \geq s_j,
\\
0 \quad &\text{for} \quad t \leq s_{j-1},
\end{aligned}
\right.
\quad |\eta_j'|\le 2m/T.
$$
Similar to the proof Proposition \ref{prop2},
we see that $u\eta_j\in \bH^{\alpha,2}_{p,q,w,0}\left((s_{j-1},s_{j+1}) \times \bR^d\right)$ and satisfies
\begin{equation}
							\label{eq0302_01}
-\partial_t^\alpha(u \eta_j) + a^{ij} D_{ij}(u \eta_j) + b^i D_i (u \eta_j) + c(u \eta_j)
= f \eta_j + h_j
\end{equation}
in $(s_{j-1},s_{j+1}) \times \bR^d$,
where $\partial_t^\alpha = \partial_t I_{s_{j-1}}^{1-\alpha}$ and
$$
h_j(t,x) = \frac{\alpha}{\Gamma(1-\alpha)} \int_{-\infty}^t(t-s)^{-\alpha-1}\left( \eta_j(s) - \eta_j(t) \right) u(s,x) \, ds.
$$
Since $\eta_j(t) = 1$ for $t \geq s_j$,
$$
h_j(t,x) = \frac{\alpha}{\Gamma(1-\alpha)} \int_{-\infty}^t(t-s)^{-\alpha-1}\left( \eta_j(s) - \eta_j(t) \right) u(s,x) \chi_{s \leq s_j} \, ds
$$
for $t \in (s_{j-1},s_{j+1})$.
In particular,
$$
h_0(t,x) = 0,
$$
and, for $j = 1,2,\ldots,m-1$,
\begin{align*}
&h_j(t,x) = \frac{\alpha}{\Gamma(1-\alpha)} \int_0^t(t-s)^{-\alpha-1}\left( \eta_j(s) - \eta_j(t) \right) u(s,x) \chi_{s \leq s_j} \, ds\\
&\leq \frac{2m}{T} \frac{\alpha}{\Gamma(1-\alpha)} \int_0^t (t-s)^{-\alpha}|u(s,x)| \chi_{s \leq s_j} \, ds
= \frac{2m\alpha}{T} I_0^{1-\alpha} | u(\cdot,x) \chi_{\cdot \leq s_j}|(t).
\end{align*}
By Lemma \ref{lem0226_1}, for $j=1,2,\ldots,m-1$,
\begin{align*}
&\|h_j\|_{p,q,w,(s_{j-1},s_{j+1})} \leq \|h_j\|_{p,q,w,(0,s_{j+1})}\\
&\leq \frac N T\left(\frac{(j+1)T}{m}\right)^{1-\alpha} \|u(t,x)\chi_{t \leq s_j}\|_{p,q,w,(0,s_{j+1})}
\leq N T^{-\alpha}\|u\|_{p,q,w,(0,s_j)}.
\end{align*}
Here and in the sequel we denote $\|\cdot\|_{p,q,w,(\tau_1,\tau_2)} = \|\cdot\|_{L_{p,q,w}\left((\tau_1,\tau_2) \times \bR^d\right)}$.
This estimate combined with Lemma \ref{lem0226_2} and Corollary \ref{cor0221_01} applied to \eqref{eq0302_01} shows that
\begin{align*}
&\|u\|_{p,q,w,(s_j, s_{j+1})} \leq \|u \eta_j\|_{p,q,w,(s_{j-1}, s_{j+1})}
\leq N \left(\frac{T}{m}\right)^\alpha \|\partial_t^\alpha (u\eta_j)\|_{p,q,w,(s_{j-1}, s_{j+1})}\\
&\leq N_0 \left(\frac{T}{m}\right)^\alpha
\|f\eta_j\|_{p,q,w,(s_{j-1}, s_{j+1})} + N_0m^{-\alpha} \|u\|_{p,q,w,(0,s_j)}\\
&\quad + N_1 \left(\frac{T}{m}\right)^\alpha \|u\|_{p,q,w,(s_{j-1}, s_j)} + N_1 \left(\frac{T}{m}\right)^\alpha \|u\|_{p,q,w,(s_j, s_{j+1})},
\end{align*}
where $N_0 = N_0(d,\delta,\alpha,p,q,K_1)$ and $N_1 = N_1(d,\delta,\alpha,p,q,K_1,K_0,R_0)$.
By taking a sufficiently large integer $m$ so that
$$
N_1\left(\frac{T}{m}\right)^\alpha < \frac{1}{2},
$$
we see that
$$
\|u\|_{p,q,w,(s_j,s_{j+1})} \leq N \|f\|_{p,q,w,(0,s_{j+1})} + N \|u\|_{p,q,w,(0,s_j)},
$$
where $N = N(d,\delta,\alpha,p,q,K_1,K_0,R_0,T)$ and $j=0,1,\ldots,m-1$.
Upon noting that $\|u\|_{p,q,w,(0,s_0)} = 0$ and using induction, we arrive at \eqref{eq0226_02}.

To prove the existence result, one can use the results in \cite{MR3899965} for the unmixed case without weights and the argument in \cite[Section 8]{MR3812104}, or alternatively use the a prior estimate proved above and the solvability of a simple equation presented in \cite{arXiv:1911.07437}.
\end{proof}

\bibliographystyle{plain}

\def\cprime{$'$}

\end{document}